\documentclass{article}
\pdfoutput=1
\usepackage{float}
\usepackage[english]{babel}
\usepackage[margin=1in]{geometry}
\usepackage{background}
\usepackage{amssymb}
\usepackage{amsthm}
\usepackage{dsfont}
\usepackage{amsmath, amsfonts, accents}
\usepackage{scrextend}
  \usepackage{paralist}
  \usepackage{graphics} 
  \usepackage{epsfig} 
\usepackage{graphicx}
\usepackage{epstopdf}
\usepackage[font={small,it}]{caption}
\usepackage{subcaption}
\graphicspath{{figures/}} 

 \usepackage[colorlinks=true]{hyperref}
\hypersetup{urlcolor=blue, citecolor=red}

\SetBgContents{}

\DeclareMathOperator{\Id}{Id}

\DeclareMathOperator{\dist}{dist}

\DeclareMathOperator*{\esssup}{ess\,sup}

\DeclareMathOperator*{\essinf}{ess\,inf}
\DeclareMathOperator{\Var}{Var}

\DeclareMathOperator{\rk}{rk}

\newtheorem{theorem}{Theorem}[section]
\newtheorem*{theorem-non1}{Theorem 1.1}
\newtheorem*{theorem-nonA}{Theorem A}
\newtheorem*{theorem-nonB}{Theorem B}
\newtheorem*{theorem-nonC}{Theorem C}

\newtheorem{lemma}[theorem]{Lemma}
\newtheorem{proposition}[theorem]{Proposition}
\newtheorem{conjecture}[theorem]{Conjecture}
\theoremstyle{definition}
\newtheorem{definition}[theorem]{Definition}

\newtheorem{assumption}[theorem]{Assumption}
\newtheorem{definition and lemma}[theorem]{Definition and Lemma}

\newcommand{\rmd}{\mathrm{d}}
\newcommand{\rmD}{\mathrm{D}}
\renewcommand{\epsilon}{\varepsilon}

\newcommand{\fa}          {\quad \text{for all } \,}

\newcommand{\R} {\mathbb R}

\numberwithin{equation}{section}

\begin{document}

\title{\vspace*{-10mm}
Conditioned Lyapunov exponents for random dynamical systems}
\author{Maximilian Engel\thanks{Zentrum Mathematik der TU M\"{u}nchen,
Boltzmannstr. 3, D-85748 Garching bei M\"{u}nchen} \and Jeroen S.W.~Lamb\thanks{Department of Mathematics, Imperial College London, 180 Queen’s Gate, London SW7 2AZ, United Kingdom
} \and Martin~Rasmussen\footnotemark[2]}
\date{\today}
\maketitle

\bigskip

\begin{abstract}
We introduce the notion of Lyapunov exponents for random dynamical systems, conditioned to
trajectories that stay within a bounded domain for asymptotically long times. This is motivated by
the desire to characterize local dynamical properties in the presence of unbounded noise (when almost all trajectories are unbounded).
We illustrate its use in the analysis of local bifurcations in this context.

The theory of conditioned Lyapunov exponents of stochastic differential equations builds on the stochastic analysis of quasi-stationary distributions for killed processes and associated quasi-ergodic distributions. We show that conditioned Lyapunov exponents describe the local stability behaviour of trajectories that remain within a bounded domain and -- in particular -- that negative conditioned Lyapunov exponents imply local synchronisation. Furthermore, a conditioned dichotomy spectrum is introduced and its main characteristics are established.

\bigskip

\noindent $\textit{Keywords}$. Quasi-stationary distribution, Quasi-ergodic distribution, Lyapunov exponent, Random dynamical system, Stochastic bifurcation.

\noindent $\textit{Mathematics Subject Classification (2010)}$.
37A50, 37H10, 37H15, 60F99.

\end{abstract}

\section{Introduction} \label{intro}
Lyapunov exponents are central to the theory of dynamical systems, providing a quantitative measure of local instability that underlies the celebrated sensitive dependence on initial conditions typifying deterministic chaos. Increasingly, mathematical modelling concerns the evaluation of stochastic differential equations consisting of a deterministic (nonlinear) drift and an unbounded white noise diffusion. In this stochastic setting, the theory of Lyapunov exponents remains well established, but it characterizes only global dynamical properties as the unbounded noise enforces a unique ergodic component that is equal to the entire phase space: under a perturbation by unbounded noise, all attractors and other flow-invariant objects of a deterministic dynamical system are joined together. This is one of the main reasons why it has been difficult to extend the bifurcation theory for deterministic dynamical systems, describing qualitative changes in dynamical behaviour, to the stochastic context with unbounded noise, since the local building blocks of bifurcation theory cease to exist.\footnote{We note that random dynamical systems with bounded noise suffer less from this problem, but such systems are generally not amenable to techniques from stochastic analysis and we will not consider these here. We refer to \cite{bhy12,hyg2013,Lamb_15_1,Zmarrou_07_1} for alternative approaches to bifurcations of random dynamical systems in the bounded noise context.}


The objective of this paper is to develop a notion of Lyapunov exponents in the context of killed processes. A Markov process is said to be killed if its domain contains certain traps from which the process cannot escape (at which point the process is considered to be killed). The aim is to study the Markov process under the condition that it survives for asymptotically long times. It is normally assumed in this context that trajectories of the killed Markov process hit a trap almost surely in finite time. In this paper, we consider the Markov process to be induced by an SDE 
with additive unbounded white noise on a bounded domain $E \subset \mathbb{R}^d$:
\begin{equation} \label{SDErd_intro}
\rmd X_t = f(X_t) \, \rmd t + \sigma \, \rmd W_t\,, \ X_0 = x \in E\,,
\end{equation}
where the drift $f$ is continuously differentiable. We consider trajectories of the SDE starting inside the interior of a bounded subset $E$ of $\mathbb{R}^d$ conditioned on the fact that they do not reach the boundary $\partial E$. In other words, the above mentioned trap is constituted by $\partial E$.

The theory of conditioned processes goes back to the pioneering work of Yaglom in 1947 \cite{yag47}, but in recent years, new ideas have been developed (see \cite{cmm13, mv12} for recent surveys). Due to the loss of mass by absorption at the boundary, the existence of a stationary distribution is impossible and, therefore, stationarity is replaced by quasi-stationarity. A \textit{quasi-stationary distribution} preserves mass along the process conditioned on survival. Analogously, ergodicity is a potentially problematic concept in this context. Nonetheless, given a unique quasi-stationary distribution for a Markov process $(X_t)_{t \geq 0}$ on a state space $E$, one can derive the existence of a \textit{quasi-ergodic distribution} $m$ \cite{bro99}, characterized by
\begin{equation*}
\lim_{t \to \infty} \mathbb{E}_{x} \left( \frac{1}{t} \int_{0}^t f(X_s) \, \rmd s \bigg| T>t \right) = \int_{E} f \, \rmd m
\end{equation*}
for all measurable and bounded observables $f$ and initial points $x$. Here the time of absorbtion at the boundary of $E$ is denoted by $T$.

Building on recent results by Villemonais, Champagnat, He and others \cite{CV16, ccv16, hzz16}, we obtain as our main result the existence of a \textit{conditioned Lyapunov exponent}:
\begin{theorem-nonA}
Let $(\theta, \varphi)$ be the random dynamical system with absorption at the boundary corresponding to the Markov process $(X_t)_{t \geq 0}$ solving equation~\eqref{SDErd_intro}. If $(X_t)_{t \geq 0}$ and the projection of its derivative process possesses a joint quasi-ergodic distribution, then for all $v \in \mathbb{R}^d \setminus \{0\}$ and $x \in E$, the conditioned expectation of the finite-time Lyapunov exponents converges to the so-called conditioned Lyapunov exponent  $\lambda$:
\begin{equation*}
\lambda = \lim_{t \to \infty} \frac{1}{t} \mathbb{E}_x \left[ \ln \frac{\| \rmD \varphi(t,\cdot,x) v \|}{\|v\|} \bigg|  T > t \right]\,.
\end{equation*}
If $d=1$, the domain $E$ is an interval $E=I \subset \mathbb{R}$, and the conditioned Lyapunov exponent $\lambda$ is given by
\begin{equation} \label{lambda_onedim_intro}
\lambda = \int_I f'(y) \, m(\rmd y)\,,
\end{equation}
where $m$ denotes the quasi-ergodic distribution for~\eqref{SDErd_intro}.
\end{theorem-nonA}
The proof (see Theorem~\ref{fk_killed}) relies on the fact that the finite-time Lyapunov exponents can be expressed as the time averages of a functional. These time averages conditioned on survival converge to an integral with respect to the relevant quasi-ergodic distribution. Existence of the quasi-ergodic distribution follows from standard theory \cite{ccv16, cmm13, hzz16} when $d=1$ and from an assumption similar to \cite[Assumption (A')]{CV16} when $d \geq 2$.

We furthermore find that the conditioned Lyapunov exponent $\lambda$ measures the  asymptotic (in)stability of typical surviving trajectories in the sense that the finite-time Lyapunov exponents
$$\lambda_v(t,\cdot,x) := \frac{1}{t}  \ln \frac{\| \rmD \varphi(t,\cdot,x) v \|}{\|v\|}\,.$$
converge in probability to $\lambda$.
%
\begin{theorem-nonB}
Let $\lambda_v(t,\cdot,x)$ denote the finite-time Lyapunov exponents associated with equation~\eqref{SDErd_intro}.
Then for all $\epsilon > 0$, we have
\begin{equation*}
\lim_{t  \to \infty} \mathbb{P}_x \left( \left| \lambda_v(t,\cdot,x) - \lambda \right| \geq \epsilon \big| T > t \right) = 0
\end{equation*}
uniformly over all $x \in E$ and $v \in \mathbb S^{d-1}$.
\end{theorem-nonB}
This result is nontrivial since $\lambda$ is only defined as a limit of expectations. We moreover note that convergence in probability is the strongest result possible as almost sure convergence cannot be achieved in this setting. 

As another indication of its dynamical relevance, we obtain that a negative conditioned Lyapunov exponent $\lambda$ implies local synchronization.
\begin{theorem-nonC}
If $\lambda < 0$, then there is exponentially fast local synchronization of trajectories with arbitrarily high probability, i.e.~for all $0 < \rho < 1$, $x\in E$ and $\lambda_{\epsilon} \in (\lambda,0)$ there exists  an $\alpha_x>0$ such that
$$
\lim_{t \to \infty}
\mathbb{P}_x \left( \tfrac{1}{t} \ln \left\|
\varphi(t,\cdot,x)-\varphi(t,\cdot,y)\right\| \leq \lambda_{\epsilon} \mbox{ for all }y \in B_{\alpha_x}(x)
\big|  T > t \right) > 1 - \rho\,. $$
\end{theorem-nonC}

Finally, in Section~\ref{expdich_sec} we define the notion of a conditioned dichotomy spectrum in the above context and show
that it consists of a finite number $ n\leq d$ 
of closed intervals (Theorem~\ref{Dichtheorem}).
We moreover extend a relation between the dichotomy spectrum and the spectrum of finite-time Lyapunov exponents for SDEs, first obtained in \cite{cdlr16}, to the conditioned setting (Theorem~\ref{boundaryDichotomy}).

Our paper initiates the use of conditioned measures for killed processes to study local properties of random dynamical systems and many interesting problems remain open. For instance, it is unclear whether -- in analogy to the global setting -- one can define a spectrum of conditioned Lyapunov exponents, where the conditioned Lyapunov exponent as defined here corresponds to the maximum. One may further address the existence of other conditioned ergodic quantities, such as metric entropy, and conditioned versions of other facts, such as he correspondence between invariant measures of random dynamical systems and stationary measures of the associated Markov processes. It would also be of interest to compare conditioned dynamical quantities to properties of dynamical systems with bounded noise, for instance in the setting of Hopf bifurcation, cf.~\cite{bhy12,delr17}.


Our results are motivated by our interest in the development of a local bifurcation theory for random dynamical systems with unbounded noise. We conclude this introduction with an example to illustrate the use of our results in this context.

\subsection*{Example: local pitchfork bifurcation with additive noise}

Consider the one-dimensional SDE of the form \eqref{SDErd_intro} with parametrized drift term
\begin{equation}\label{pitchfork_drift}
f_{\alpha}(x) = \alpha x-x^3\,,
\end{equation}
where $\alpha \in \mathbb{R}$.
In the absence of noise, when $\sigma=0$, the resulting ODE is
the usual normal form for a pitchfork bifurcation.
The bifurcation at $\alpha =0$ implies the appearance of a local instability of the equilibrium at the origin
entailing a change of the attractor from $\{0\}$ for $\alpha \leq 0$ to $[-\sqrt{\alpha}, \sqrt{\alpha}]$ for $\alpha >0$.

In the presence of noise, when $\sigma>0$, it was shown in \cite{cf98} that the Lyapunov exponent is always negative.
Consequently there is a unique globally attracting random fixed point for all values of $\alpha$, leading \cite{cf98} to conclude that the noise destroys the pitchfork bifurcation. However, \cite{cdlr16} have shown that more subtle dynamical changes take place when $\alpha$ increases through $0$:  the random attractor loses its uniform attractivity, which is signalled by a zero-level crossing of the dichotomy spectrum.

The negativity of the Lyapunov exponent refers to a global property and relies on the global dominance of contractive properties of the flow.
Figure~\ref{clnum} shows how the conditioned Lyapunov exponent is able to reveal the local instability around the unstable deterministic equilibrium at the origin. In Figure~\ref{change_in_c}, as $c$ increases the contractivity of the dynamics in the deterministic basins of the attractive equilibria is increasingly felt, leading do a decrease in $\lambda$. Similarly, in Figure~\ref{change_in_alpha}, given a fixed neighbourhood of the origin, with increasing $\alpha$, we observe a loss of local stability.

The local approach is also convenient if $f_\alpha$ is on a global level more complicated. For instance, if $f_{\alpha}(x) = \alpha x-x^3+0.3x^5$ the SDE has no global stationary measure and the global Lyapunov exponent is not defined. But by focussing on a local
domain near the origin, say $[-1.5,1.5]$ one finds local dynamical properties well represented by the quasi-ergodic measure and conditioned Lyapunov exponents, similar to when the highest order term is not present. Moreover, in case we add another seventh order term to globally stabilize the dynamics, so that for instance $f_{\alpha}(x) = \alpha x-x^3+0.3x^5-0.1x^7$, a global stationary measure and Lyapunov exponent exist, but due to the presence of unstable equilibria far from the origin, there will be no bifurcation in the dichotomy spectrum when $\alpha$ passes through zero. By conditioning to a local neighbourhood of the origin, e.g.~$[-1.5,1.5]$, one regains the correspondence between a local
loss of stability and loss of attractivity in the conditioned dichotomy spectrum.

\begin{figure}[H]
\centering
\begin{subfigure}{.25\textwidth}
  \centering
  \includegraphics[width=1\linewidth]{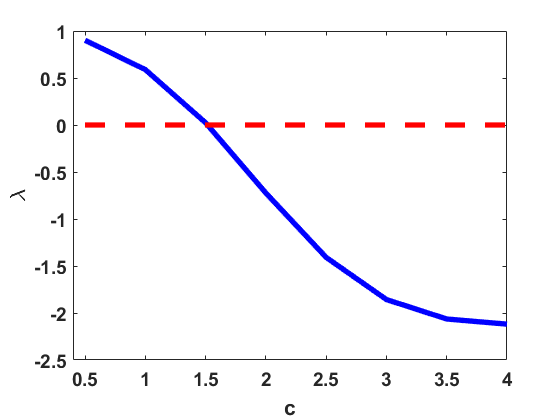}
  \caption{$\lambda(c)$  when $\alpha=1$}
  \label{change_in_c}
\end{subfigure}%
 \hspace*{10em}
\begin{subfigure}{.25\textwidth}
  \centering
  \includegraphics[width=1\linewidth]{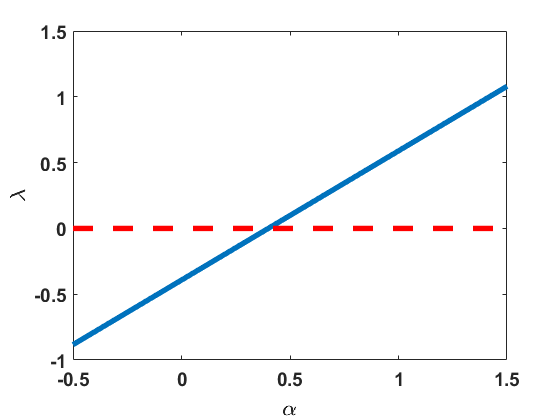}
  \caption{$\lambda(\alpha)$ when $c=1$}
  \label{change_in_alpha}
\end{subfigure}
\caption[ ]{\label{clnum} Conditioned Lyapunov exponent $\lambda$ for SDE \eqref{SDErd_intro} with drift \eqref{pitchfork_drift} and $\sigma=1$ on the domain $[-c,c]$, obtained from a numerical approximation of the quasi-ergodic distribution $m$ in~\eqref{lambda_onedim_intro}.
The dashed red line indicates the zero level of $\lambda$ to fascilitate the observation of its sign change.}
\label{Lyapexp}
\end{figure}



\section{Quasi-stationary and quasi-ergodic distributions} \label{general}

Let $E \subset \mathbb{R}^d$ be an open and bounded connected set with $C^2$-boundary $\partial E$, and consider the stochastic differential equation
\begin{equation} \label{MultidimSDE}
  \rmd X_t = f(X_t) \rmd t + \sigma \rmd W_t\,,
\end{equation}
where $f: E \to \mathbb{R}^d$ is continuously differentiable, $\sigma>0$, and $(W_t)_{t\in\R}$ denotes some $d$-dimensional standard Brownian motion. We assume that $f$ can be extended continuously on $\bar E$.

For an initial condition $X_0 \in E$, we consider the time-homogeneous Markov process $(X_t)_{t \geq 0}$ as solution to \eqref{MultidimSDE}. Let $\mathcal E := \mathcal B (E \cup \partial E)$ be the Borel $\sigma$-algebra. Then the process $(X_t)_{t \geq 0}$ is associated with a family of probabilities $( \mathbb{P}_x)_{x \in E}$ on the Wiener space $(\Omega, \mathcal{F}, \mathbb P)$. We have
$$ \mathbb{P}_x(X_0 = x) =1 \fa x \in E \cup \partial E\,,$$
and the transition probabilities $(\hat P_t)_{t\geq 0}$ are given by
$$
  \hat P_t(x,A) = \mathbb{P}_x (X_t \in A) \fa x \in E \cup \partial E \mbox { and } A \in \mathcal E\,.
$$
The process is further associated with a semi-group of operators $(P_t)_{t\geq 0}$ given by
$$ P_t g(x) = \mathbb{E}_x [g(X_t)]$$
for all measurable and bounded functions $g : E\cup\partial E\to\R$. We consider the Markov process to be killed at $\partial E$, i.e.~$X_s \in \partial E$ for some $s\ge0$ implies $X_t = X_s$ for all $t \geq s$. This means  that the random variable
$$ T := \inf \big\{t \geq 0: X_t \in \partial E \big\}$$
is a stopping time, and we have $X_t = X_T$ for all $t \geq T$.
Note that killed Markov processes induced by \eqref{MultidimSDE} satisfy that for all $x \in E$ and $t\ge 0$, we have
\begin{equation} \label{finitekilling}
T < \infty \quad\mathbb{P}_x\text{-almost surely} \qquad \mbox{and} \qquad \mathbb{P}_{x} (T > t) > 0\,.
\end{equation}

In our context, there are no stationary measures on $E$, since the Markov process $(X_t)_{t\ge0}$ is killed at the boundary. However, quasi-stationary measures often exist.

\begin{definition}[Quasi-stationary distribution]
A \emph{quasi-stationary distribution} (QSD) is a probability measure $\nu$ on $E$ such that for all $ t \geq 0$ and all measurable sets $B \subset E$
\begin{equation} \label{QSDdef}
 \mathbb{P}_{\nu}\left( X_t \in B | T > t \right) = \nu(B)\,.
\end{equation}
\end{definition}
\noindent Here, we use the notation $ \mathbb{P}_{\mu} = \int_{E} \mathbb{P}_x \,  \mu(\rmd x)$ for any probability measure $\mu$ on $E$.

It follows that $T$ is exponentially distributed for a process started with a QSD $\nu$ \cite{cmm13, fkmp95}.
\begin{proposition} \label{exponential_escape}
If $\nu$ is a QSD, then there exists a $\lambda_0 < 0$ such that for all $t \geq 0$,
$$\mathbb{P}_{\nu}(T > t) = e^{\lambda_0 t}\,.$$
We call $\lambda_0$ the (exponential) survival rate and $-\lambda_0$ the \textit{(exponential) escape rate} associated with the quasi-stationary distribution $\nu$.
\end{proposition}
%
A particular class of quasi-stationary distributions are \emph{quasi-limiting distributions}, which satisfy for all bounded and measurable functions $h: E \to \mathbb{R}$ that
\begin{equation} \label{qsd_xconverge}
  \lim_{t \to \infty} \mathbb{E}_{x}[h(X_t)| T > t] = \int_{E} h(y)  \nu(\rmd y) \fa x\in E\,.
\end{equation}
However, as we aim at studying ergodic quantities such as Lyapunov exponents, we are more interested in time averages. This motivates the following definition, similarly to \cite{bro99}.
\begin{definition}[Quasi-ergodic distribution]
A probability measure $m$ on $E$ is called \emph{quasi-ergodic distribution} (QED) if for all $t > 0$, every bounded and measurable function $h: E \to \mathbb{R}$ and every $x \in E$, the following limit exists and satisfies
\begin{equation} \label{QEDdef}
\lim_{t \to \infty} \mathbb{E}_{x} \left[ \frac{1}{t} \int_{0}^t h(X_s) \rmd s \bigg| T>t \right] = \int_{E} h \rmd m\,.
\end{equation}
\end{definition}

Champagnat et al.~\cite{ccv16} can then prove a result which immediately implies the following:
\begin{theorem}[QSD and QED for stochastic differential equations] \label{SDE_CCV}
Let $(X_t)_{t\geq 0}$ be the Markov process induced by the stochastic differential equation~\eqref{MultidimSDE} with initial condition $X_0\in E$.
Then the following statements hold.
\begin{enumerate}[(a)]
\item  There exists a QSD $\nu$ with a survival rate $\lambda_0$, which fulfills for $C>0$, $\gamma >0$ and all probability measures $\mu$ on $E$ that
\begin{equation}\label{qsdexpconv}
  \| \mathbb{P}_{\mu} (X_t \in \cdot | T > t) - \nu(\cdot) \|_{TV} \leq C e^{- \gamma t} \fa t \geq 0 \,,
\end{equation}
where $\| \mathbb{P} - \mathbb{Q} \|_{TV} := \sup_{A \in \mathcal E} \left| \mathbb{P}(A) - \mathbb{Q}(A) \right| $ denotes the total variation distance for probability measures.
\item There exists  a QED $m$ given by
$$ m(\rmd x) = \eta(x) \nu(\rmd x)\,,$$
where
\begin{equation}\label{etadef}
 \eta (x) = \lim_{t \to \infty} \frac{\mathbb{P}_x(T > t)}{\mathbb{P}_{\nu}(T > t)} = \lim_{t \to \infty} e^{-\lambda_0 t} \mathbb{P}_x (T > t)\,
\end{equation}
is a bounded eigenfunction of the generator $\mathcal L$ of the semi-group $(P_t)_{t\ge0}$ for the eigenvalue $\lambda_0$. The convergence to the QED $m$ in \eqref{QEDdef} is uniform over all $x \in E$.
\end{enumerate}
\end{theorem}
\begin{proof}
In \cite{CV16}, three equivalent conditions for exponential convergence to a quasi-stationary distribution are considered. To obtain (a), we use the following condition, which is satisfied for the stochastic differential equation \eqref{MultidimSDE}, as shown in \cite{ccv16}: there exists a family of probability measures $(\nu_{x_1, x_2})_{x_1, x_2 \in E}$ on $E$ such that
\begin{itemize}
\item[(A1)] there exist $t_0, c_1 > 0$ such that for all $x_1, x_2 \in E$,
$$\mathbb{P}_{x_i} (X_{t_0} \in \cdot | T > t_0) \geq c_1 \nu_{x_1, x_2}(\cdot) \fa i=1,2\,,$$
\item[(A2)] and there exists $c_2 > 0$ such that for all $x_1, x_2 \in E$ and $t \geq 0$,
$$\mathbb{P}_{\nu_{x_1, x_2}} (T > t) \geq c_2 \sup_{x \in E} \mathbb{P}_{x} (T > t)\,.$$
\end{itemize}
The existence of the QED $m$ with the characterization of $\eta$ as given in (b) follows from \cite{hzz16}, which combines results from \cite{CV16} and \cite{bro99}.
\end{proof}

\section{Conditioned Lyapunov exponents and synchronization} \label{msde}

In the following, we investigate killed processes from a random dynamical systems perspective. We prove the existence of a Lyapunov exponent corresponding to the quasi-ergodic measure, and we show convergence of finite-time Lyapunov exponents to the Lyapunov exponent.

We consider the stochastic differential equation \eqref{MultidimSDE} and the time-homogeneous Markov process $(X_t)_{t\geq 0}$ for an initial condition $X_0\in E$. Let $(\theta, \varphi)$ be the random dynamical system generated by \eqref{MultidimSDE} (see \cite{a98}). We assume that $\varphi: \mathbb{R}_{0}^+ \times \Omega \times \bar E  \to \bar E$ is globally defined in time, in the sense that it takes a constant value in $\partial E$ if the system is killed at the boundary $\partial E$.
We have
$$
  \mathbb P_x( X_t \in B) = \mathbb{P}(\varphi(t, \cdot, x) \in B) \fa  t \geq 0\,,\, x \in E \mbox{ and } B \in \mathcal B (\bar E)\,.
$$
Define the stopping time $\tilde T: \Omega \times E \to \mathbb{R}_0^+$ as
$$ \tilde T(\omega,x) = \inf \big\{ t > 0\,:\, \hat{\varphi}(t, \omega,x) \in \partial E \big\}\,$$
such that for all $x \in E$ and $t\geq 0$
$$ \mathbb P_x( T > t) = \mathbb{P}( \tilde T(\cdot,x) > t)\,.$$

Note that $\varphi(t, \omega,\cdot)$ is differentiable for all $\omega \in\Omega$, $x\in E$ and $t < \tilde T(\omega,x)$. We consider the finite-time Lyapunov exponents
\begin{displaymath}
  \lambda_{v}(t,\omega,x) = \frac{1}{t}  \ln \frac{\| \rmD \varphi(t,\omega,x) v \|}{\|v\|} \fa t \in \big(0, \tilde T(\omega,x)\big)\,,
\end{displaymath}
where $\rmD \varphi$ solves the variational equation corresponding to \eqref{MultidimSDE} given by
\begin{equation} \label{varequ}
\rmd Y(t,\omega,x) = \rmD f(\varphi(t,\omega,x)) Y(t,\omega,x) \, \rmd t \,,
\end{equation}
where $Y(0,\omega,x) = \Id$.

\subsection{Conditioned Lyapunov exponent}

We now investigate the behavior of finite-time Lyapunov exponents in the limit $t\to\infty$, conditioned to non-absorption at the boundary.
We use Furstenberg--Khasminskii averaging to show that
\begin{equation} \label{lambda_v}
\lambda := \lim_{t \to \infty} \mathbb{E} \big[ \lambda_v (t, \cdot, x)  \big| \tilde T(\cdot,x) > t \big] \fa x\in E \mbox{ and } v\in \R^d\setminus \{0\}
\end{equation}
exists and is independent of $x$ and $v$.

We have seen in the proof of Theorem~\ref{SDE_CCV} that the stochastic process $(X_t)_{t \geq 0}$ generated by \eqref{MultidimSDE} for an initial condition $X_0\in E$ satisfies (A1) and (A2) with a family of probability measures $(\nu_{x_1, x_2})_{x_1, x_2 \in E}$ and $t_0>0$. Instead of $(X_t)_{t\ge0}$, we consider the extended process $(X_t, s_t)_{t \geq 0}$, where
$$s_t(\omega, X_0) = \frac{\rmD \varphi(t,\omega,X_0)}{\| \rmD \varphi(t,\omega,X_0) \|}$$
denotes the induced process on the unit sphere. We need that (A1) and (A2) are satisfied for this process as well, in order to obtain results analogous to Theorem~\ref{SDE_CCV}. Since (A2) is automatically fulfilled, we require the following assumption.

\begin{assumption}\label{assextsys}
There exists a $c_0 > 0$ and a family of probability measures $(\mu_{z_1, z_2})_{z_1,z_2\in  E \times \mathbb S^{d-1}}$ such that for any $z_i =(x_i,s_i) \in E \times \mathbb S^{d-1}$, where $i=1,2$, and $A \in \mathcal B (E)$ with $\nu_{x_1,x_2}(A) > 0$, we have
$$ \mathbb{P}_{z_i} (s_{t_0} \in \cdot\, | X_{t_0} \in A, T > t_0) \geq c_0 \mu_{z_1,z_2} (\cdot)\,.$$
\end{assumption}

We will show below that it follows under Assumption~\ref{assextsys} that process $(X_t, s_t)_{t \geq 0}$ that possesses a joint quasi-ergodic distribution $\tilde{m}$ on $E\times \mathbb S^{d-1}$.
The following theorem is a more detailed version of Theorem A.

\begin{theorem}[Conditioned Lyapunov exponent] \label{fk_killed}
  Consider the process $(X_t, s_t)_{t \geq 0}$ under the assumption that it possesses a joint quasi-ergodic distribution $\tilde{m}$ on $E\times \mathbb S^{d-1}$. Then the \emph{conditioned Lyapunov exponent} $\lambda$ as defined in~\eqref{lambda_v} exists and is given by
  \begin{equation}\label{fkformula_quasi}
  \lambda = \lim_{t \to \infty} \mathbb{E} \big[ \lambda_v (t, \cdot, x)  \big| \tilde T(\cdot,x) > t \big] = \int_{\mathbb S^{d-1} \times E}  \langle s, \rmD f(y) s \rangle \ \tilde{m}(\rmd s, \rmd y)\,,
  \end{equation}
  where the convergence is uniform over all $x \in E$ and $v \in \mathbb{R}^d \setminus \{0\}$. A sufficient condition for existence of the quasi-ergodic distribution $\tilde m$ is given by Assumption~\ref{assextsys}.
\end{theorem}

\begin{proof}
Note that the angular component $s_t$ as defined above lies on the unit sphere $\mathbb S^{d-1}$ and write  $r_t(\omega, x) = \|\rmD \varphi(t,\omega,x) \|$ for the radial component. For fixed $\omega\in\Omega$ and $x\in E$, the variational equation~\eqref{varequ} in vector polar coordinates is given by
\begin{align*}
\rmd s_t &=  \rmD f(\varphi(t,\omega,x)) s_t - \langle s_t, \rmD f(\varphi(t,\omega,x)) s_t \rangle s_t\, \rmd t\,,\\
\rmd r_t &= \langle s_t, \rmD f(\varphi(t,\omega,x))s_t \rangle r_t\, \rmd t\,,
\end{align*}
where $s_0\in\mathbb S^{d-1}$ and $r_0\in\R^+$.
We obtain for all $\omega \in \Omega$ and $x \in E$ that
\begin{equation*}
r_t(\omega,x) = r_0 \exp \left( \int_0^t  h(\varphi(\tau, \omega,x),s_{\tau}(\omega,x) ) \, \rmd \tau \right) \,,
\end{equation*}
where $h: E \times \mathbb S^{d-1} \to \mathbb{R}$ is given by
\begin{equation*}
h(x,s) = \langle s, \rmD f(x)s \rangle\,.
\end{equation*}
We observe that $(X_t, s_t)_{t \geq 0}$ constitutes a skew product system on $E \times \mathbb S^{d-1}$ with killing at $\partial E \times \mathbb S^{d-1}$.
Let us first assume that $(X_t, s_t)_{t \geq 0}$ admit unique QSD $\tilde \nu$ and QED $\tilde m$ on $E \times \mathbb S^{d-1}$, which due to the skew product structure have $\nu$ and $m$ as their marginals on $E$ (see Theorem~\ref{SDE_CCV} for existence of $\nu$ and $m$). Hence, by definition of a quasi-ergodic distribution and the fact that $h$ is bounded and measurable by the assumptions, we conclude that for all $v \in \mathbb{R}^d$ with $v = \|1\|$ (which is enough due to linearity) and $x \in E$
\begin{align*}
\lambda &= \lim_{t \to \infty} \frac{1}{t} \mathbb{E}_{x,s_0}\big[\ln r_t | T > t\big] = \lim_{t \to \infty} \frac{1}{t} \mathbb{E}_x\big[\ln r_t | T > t\big] \\
&=  \lim_{t \to \infty}\frac{1}{t} \mathbb{E} \left[ \int_0^t h(\varphi(\tau, \cdot,x),s_{\tau}(\cdot,x) ) \, \rmd \tau \bigg| \tilde T(\cdot,x) >t \right] = \int h \, \rmd \tilde m\,,
\end{align*}
where the convergence is uniform according to Theorem~\ref{SDE_CCV}~(b).

Let us now assume Assumption~\ref{assextsys}. In order to derive existence of a unique QSD $\tilde \nu$ and associated QED $\tilde m$ for $(X_t, s_t)_{t \geq 0}$ on $E \times \mathbb S^{d-1}$, we need to check that (A1) and (A2) from the proof of Theorem~\ref{SDE_CCV} are satisfied. We know that $(X_t)_{t \geq 0}$ satisfies the assumption on $E \cup \partial E$, i.e.~there is a family $(\nu_{x_1,x_2})$ that fulfils (A1) and (A2) for some constants $t_0, c_1, c_2 > 0$. Assumption~\ref{assextsys} states that there exist a $c_0 > 0$ and a family of probability measures $(\mu_{z_1, z_2})$ such that for any $z_i =(x_i,s_i) \in E \times \mathbb S^{d-1}$, where $i=1,2$, and $A \in \mathcal B (E)$ with $\nu_{x_1,x_2}(A) > 0$
$$ \mathbb{P}_{z_i} (s_{t_0} \in \cdot \,| X_{t_0} \in A, T > t_0) \geq c_0 \mu_{z_1,z_2} (\cdot)\,. $$
We define the family of probability measures
$$ \tilde \nu_{z_1,z_2} (A \times B) = \nu_{x_1,x_2} (A)\mu_{z_1, z_2}(B) \quad \text{ for all measurable } A \subset E,\, B \subset \mathbb S^{d-1},\, z_1, z_2 \in E \times \mathbb S^{d-1}\,.$$
Since $(X_t)_{t \geq 0}$ and therefore $T$ are independent from $(s_t)_{t \geq 0}$, we observe that for all $z_1, z_2 \in E \times \mathbb S^{d-1}$ and measurable $A \subset E$ and $B \subset \mathbb S^{d-1}$, we have
\begin{align*}
\mathbb{P}_{z_i}( (X_{t_0},s_{t_0}) \in A \times B | T > t_0) & = \frac{\mathbb{P}_{z_i}( (X_{t_0},s_{t_0}) \in A \times B, T > t_0)}{\mathbb{P}_{x_i}( T > t_0)} \\
& =  \mathbb{P}_{z_i} (s_{t_0} \in B | X_{t_0} \in A, T > t_0) \mathbb{P}_{x_i}( X_{t_0} \in A | T > t_0) \\
& \geq c_0 \mu_{z_1,z_2} (B)
\mathbb{P}_{x_i}( X_{t_0} \in A | T > t_0) \geq  c_0 c_1 \tilde \nu_{z_1,z_2}(A\times B)\,.
\end{align*}
This shows (A1). Using again the independence of the hitting time $T$ from $s_t$, (A2) follows by observing that for all $z_1, z_2 \in E \times \mathbb S^{d-1}$,
$$ \mathbb{P}_{\nu_{z_1,z_2}}(T>t) = \int_{\mathbb S^{d-1}} \int_E \mathbb{P}_x(T > t) \, \nu_{x_1,x_2}(\rmd x)\, \mu_{z_1, z_2}( \rmd s) \geq c_2 \sup_{x \in E, s \in \mathbb S^{d-1}} \mathbb{P}_{x,s}(T>t)\,.$$
Analogous to the proof of Theorem~\ref{SDE_CCV}, we conclude that there are a unique QSD $\tilde \nu$ and associated QED $\tilde m$ on $E \times \mathbb S^{d-1}$ which due to the skew product structure have $\nu$ and $m$ as their marginals on $E$.
\end{proof}
The Lyapunov exponent from Theorem~\ref{fk_killed} measures dominant behaviour independently from the initial vector $v$ on the tangent space, and therefore does not take into account the dynamics in directions of weaker growth behaviour.
In the classical setting, one obtains a spectrum of Lyapunov exponents associated with a filtration or splitting of flow-invariant subspaces. In our setting, such invariant filtrations or subspaces cannot be obtained due to the finite survival time of a given path. However, it could be possible to find a spectrum of Lyapunov exponents following the Furstenberg--Kesten Theorem \cite[Theorem 3.3.3]{a98}.

For a time $t > 0$, consider the linearization $\Phi(t, \omega,x):= \rmD \varphi(t,\omega, x) \in \mathbb{R}^{d \times d}$ for all $(\omega,x)$ such that $\varphi(s,\omega, x) \in E$ for all $0 \leq s \leq t$. Let
$$0 < \sigma_d(\Phi(t, \omega,x)) \leq \dots \leq \sigma_1(\Phi(t, \omega,x))$$
be the singular values of $\Phi(t, \omega,x)$, i.e.~the eigenvalues of $\sqrt{\Phi^*(t, \omega,x) \Phi(t, \omega,x)}$.

\begin{conjecture}[Lyapunov spectrum] \label{MET_killed}
Let $\Phi$ be the linearization associated with the stochastic differential equation~\eqref{MultidimSDE}.
Then there are $\gamma_i \in \mathbb{R}$ such that for all $x \in E$ we have
$$\frac{1}{t} \,\mathbb{E}\big[ \ln \sigma_i(\Phi(t, \cdot,x)) \big| \tilde T(\cdot,x) > t\big] \xrightarrow{ t \to \infty} \gamma_i \fa 1 \leq i \leq d\,.$$
In this case, we can define a Lyapunov spectrum of growth rates of the surviving trajectories by denoting $\lambda_1 > \lambda_2 > \dots \lambda_p$ for the $p \in\{1,\dots, d\}$ different values of $\gamma_1 \geq \dots \geq \gamma_d$.
\end{conjecture}
We cannot use the same techniques as for Theorem~\ref{fk_killed} to prove this conjecture since $\ln \sigma_i(\Phi(t, \cdot,x))$ is not representable by a functional $h(X_t)$.
In the setting without killing, the Furstenberg--Kesten Theorem uses the subadditivity of exterior powers to show convergence of the exponents expressed as singular values. Conditioning to survival makes an analogous approach not feasible. Hence, one has to find a different strategy to show the conjetcure.

Note that the computation of the Lyapunov exponent $\lambda$ from Theorem~\ref{fk_killed} can be very difficult and costly in higher dimensions as we have to determine $\tilde m(\rmd s, \rmd x)$, the joint quasi-ergodic distribution of the original process and the derivative angular process. However, to obtain bounds on the Lyapunov exponent $\lambda$, it may be easier and cheaper to compute or approximate $m(\rmd x)$, the quasi-ergodic distribution of the original process.

In more detail, we define an \emph{upper} and a \emph{lower conditioned Lyapunov exponent} $\lambda^u$ and $\lambda^l$ by
\begin{equation} \label{upperlambda}
\lambda^u := \limsup_{t \to \infty}  \sup_{ \|v\| =1} \frac{1}{t}\, \mathbb{E} \left[ \ln \| \rmD \varphi(t,\cdot,x) v \| \big| \tilde T(\cdot,x) > t \right]\,,
\end{equation}
and
\begin{equation} \label{lowerlambda}
\lambda^l := \liminf_{t \to \infty} \inf_{ \|v\| =1}  \frac{1}{t}\, \mathbb{E} \left[ \ln \| \rmD \varphi(t,\cdot,x) v \| \big| \tilde T(\cdot,x) > t \right]\,.
\end{equation}
Defining
\begin{equation} \label{functionals}
\lambda^{+} (x) = \max_{\|r\|=1} \big\langle \rmD f(x)r,r\big\rangle \quad\mbox{and}\quad \lambda^{-} (x) = \min_{\|r\|=1} \big\langle\rmD f(x)r,r\big\rangle\,,
\end{equation}
we find the following bounds for these quantities.
\begin{proposition} \label{bounds}
Consider the Markov process $(X_t)_{t \geq 0}$ that solves the stochastic differential equation \eqref{MultidimSDE} and is killed at $\partial E$, with quasi-ergodic distribution $m$ and conditioned Lyapunov exponent $\lambda$.
Then the upper and lower conditioned Lyapunov exponents satisfy
\begin{equation*}
\int_{E} \lambda^-(x) m (\rmd x) \leq \lambda^l \leq \lambda \leq \lambda^u  \leq  \int_{E} \lambda^+(x) m (\rmd x)\,.
\end{equation*}
\end{proposition}
\begin{proof}
For any $\omega \in \Omega$, $x \in E$ and $ 0 \neq v \in \mathbb{R}^d$, define $r_t(\omega,x,v) := \frac{\rmD \varphi(t, \omega, x) v}{\| \rmD \varphi(t, \omega, x) v\|}$. We observe that
\begin{align*}
\frac{\rmd}{\rmd t} \| D  \varphi(t, \omega, x)v \|^2 &= 2 \big\langle D f( \varphi(t,\omega,x)) (t, \omega, x) v, \rmD \varphi(t, \omega, x)v \big\rangle \\
& = 2 \big\langle \rmD f( \varphi(t,\omega,x))r_t(\omega,x,v) , r_t(\omega,x,v) \big\rangle \| \rmD \varphi(t, \omega, x)v \|^2 \\
& \leq 2 \lambda^+( \varphi(t,\omega,x)) \| \rmD \varphi(t, \omega, x)v \|^2.
\end{align*}
Analogously, we obtain
$$ \frac{\rmd}{\rmd t} \| D  \varphi(t, \omega, x)v \|^2  \geq 2 \lambda^-( \varphi(t,\omega,x)) \| \rmD \varphi(t, \omega, x)v \|^2\,.$$
Hence, we can conclude that for all $ 0 \neq v \in \mathbb{R}^d$
\begin{align}
\| \rmD \varphi(t, \omega, x)v \|^2 & \leq \| v \|^2 \exp \left(2 \int_0^t \lambda^+( \varphi(s,\omega,x)) \rmd s \right)\,, \label{upperbound} \\
\| \rmD \varphi(t, \omega, x)v \|^2 & \geq \| v \|^2 \exp \left(2 \int_0^t \lambda^-( \varphi(s,\omega,x)) \rmd s \right)\,. \label{lowerbound}
\end{align}
Since $\lambda^+$ and $ \lambda^-$ are measurable and bounded on $E$, we can conclude with Theorem~\ref{SDE_CCV} that
\begin{align*}
\lambda^u & \leq \limsup_{t \to \infty}  \mathbb{E} \left[ \frac{1}{t} \int_0^t \lambda^+( \varphi(s,\cdot,x)) \rmd s \bigg| \tilde T(\cdot,x) > t \right]  \\ &= \lim_{t \to \infty}  \mathbb{E}_x \left[ \frac{1}{t} \int_0^t \lambda^+( X_s) \, \rmd s \bigg| T > t \right]
= \int_{E} \lambda^+(x)\, m (\rmd x)\,,
\end{align*}
and
\begin{align*}
\lambda^l &\geq \liminf_{t \to \infty}  \mathbb{E} \left[ \frac{1}{t} \int_0^t \lambda^-( \varphi(s,\cdot,x)) \rmd s \bigg| \tilde T(\cdot,x) > t \right]\\ &= \lim_{t \to \infty}  \mathbb{E}_x \left[ \frac{1}{t} \int_0^t \lambda^-( X_s) \, \rmd s \bigg| T > t \right] = \int_{E} \lambda^-(x) \, m (\rmd x)\,.
\end{align*}
The fact that $\lambda^l \leq \lambda \leq \lambda^u$ follows directly from the respective definitions.
\end{proof}

We aim at providing an explicit formula for the Lyapunov exponent $\lambda$ in case the stochastic differential equation \eqref{MultidimSDE} is one-dimensional. Then the problem is reduced to considering systems on an open interval $E=I=(a,b)$. In this case, the finite-time Lyapunov exponents are given by
\begin{equation*}
\lambda(t,\omega,x) = \frac{1}{t} \ln \left|\rmD \varphi(t,\omega,x) \right| \fa  t < \tilde T(\omega,x)\,,
\end{equation*}
where $ \rmD \varphi(t,\omega,x)$ solves the linear variational equation $\dot{v}(t,\omega,x) = f'(\varphi(t,\omega,x) ) v(t,\omega,x)$, so we can immediately infer that
\begin{equation}\label{varequ1d}
\lambda(t,\omega,x) = \frac{1}{t} \int_0^t f'(\varphi(s,\omega,x)) \, \rmd s  \fa t < \tilde T(\omega,x)\,.
\end{equation}
In order to compute the quasi-stationary distribution in this one-dimensional context, we follow \cite[Section~6.1.1]{cmm13} and define
$$\gamma(x) :=  \frac{2}{\sigma^2}\int_{a}^x f(y) \, \rmd y \fa x\in [a,b]$$  and the absolutely continuous measure $\mu$ with Lebesgue density $x\mapsto \exp\left( \gamma(x) \right)$
on $I$.
As in Theorem~\ref{SDE_CCV}, we consider the generator of the semigroup associated with \eqref{MultidimSDE}, given by
\begin{equation*} 
\mathcal{L} \cdot = \frac{\sigma^2}{2} \partial_{xx} \cdot + f \partial_{x} \cdot\,,
\end{equation*}
with Dirichlet boundary conditions at $x = l$ and $x = r$,
and its formal adjoint
\begin{equation*} 
\mathcal{L}^* \cdot = \frac{\sigma^2}{2} \partial_{xx} \cdot - \partial_{x} (f \cdot)\,.
\end{equation*}
With standard theory (see e.g.~\cite[Chapter 7]{cd55}), we observe that $\mathcal{L}$ is self-adjoint in $L^2(I, \mu)$ with simple eigenvalues $0$ and $\lambda_n < 0$ for $n \in\mathbb N_0$, such that the only possible accumulation point of the set $\{\lambda_n \,:\, n \in\mathbb N_0 \}$ is $-\infty$.
As in Theorem~\ref{SDE_CCV}, we denote the largest non-zero eigenvalue by $\lambda_0$ and by $\psi$ the unique solution to
$$\mathcal{L} \psi = \lambda_0 \psi\,,\quad \psi(a) = \psi(b) = 0\,,\quad  \int_I \psi^2 \rmd \mu = 1 \quad\mbox{and}\quad \psi'(a) > 0\,.$$
We further observe that $\phi(x) = \psi(x) \exp(\gamma(x))$ satisfies
$$ \mathcal{L}^* \phi = \lambda_0 \phi\,,\quad  \phi(a) = \phi(b) = 0  \quad\mbox{and}\quad \phi'(a) > 0\,.$$
In our one-dimensional setting, we now provide an explicit formula for the conditioned Lyapunov exponent $\lambda$ obtained in Theorem~\ref{fk_killed}.
\begin{proposition}[Conditioned Lyapunov exponent in one dimension] \label{Lyaponedim}
Consider the stochastic differential equation \eqref{MultidimSDE} on the interval $E=I \subset \mathbb{R}$. Then the conditioned Lyapunov exponent $\lambda$ is given by
\begin{equation} \label{lambda_onedim}
\lambda = \lim_{t \to \infty} \mathbb{E} \big[ \lambda(t, \cdot, x) \big| \tilde T(\cdot,x) > t \big] = \int_I f'(y) \, m(\rmd y) = \int_I f'(y) \psi^2(y) e^{\gamma(y)}\, \rmd y \fa x\in I\,.
\end{equation}
\end{proposition}
\begin{proof}
It follows immediately from Theorem~\ref{SDE_CCV}, by using that $f'$ is bounded and measurable on $I$ and by definition of the QED $m$, that
\begin{align*}
 \lim_{t \to \infty} \mathbb{E} \big[ \lambda(t, \cdot, x) \big| \tilde T(\cdot,x) > t \big] &\stackrel{\eqref{varequ1d}}{=} \lim_{t \to \infty} \mathbb{E} \left[ \frac{1}{t} \int_0^t f'(\varphi(s,\cdot,x)) \, \rmd s   \bigg| \tilde T(\cdot,x) > t \right] = \lim_{t \to \infty} \mathbb{E}_x \left[ \frac{1}{t} \int_0^t f'(X_s) \, \rmd s   \bigg| T > t \right]  \\
 &\stackrel{\eqref{QEDdef}}{=} \int_I f'(y) \, m(\rmd y)\,.
\end{align*}
The formula for this integral can be derived as follows. We know from \cite[Theorem 6.4]{cmm13} that
$$ \nu (\rmd x) = \frac{\phi(x) \, \rmd x}{\int_I \phi(y) \,\rmd y } = \frac{\psi(x) \mu (\rmd x)}{\int_I \psi(y) \mu (\rmd y) } \,.$$
From the fact that $\eta$ has to be proportional to $\psi$ according to Theorem~\ref{SDE_CCV} and the normalization condition on $\psi$, we obtain
$$ \eta(x) =  \left( \int_I \psi(y) \mu(\rmd y)\right) \psi(x)$$
and
$$ m(\rmd x) = \psi^2(x) \mu(\rmd x) =  \psi^2(x) \exp\left( \gamma(x) \right) \rmd x\,.$$
This finishes the proof.
\end{proof}
\subsection{Convergence to the conditioned Lyapunov exponent}

We observe that the conditioned Lyapunov exponent $\lambda$ from Theorem~\ref{fk_killed} (and given in Proposition~\ref{Lyaponedim} for one dimension) is defined as a limit of conditioned expectation. Due to the killing at the boundary, the question whether  the finite-time Lyapunov exponents $\lambda_v (t, \omega, x)$ converge to $\lambda$ as $t\to\infty$ for almost all $\omega\in\Omega$ is ill-posed. In this subsection, we prove Theorem B, i.e.~a convergence result for finite-time Lyapunov exponents in probability, which is the strongest possible form of convergence in this context.

We first need the following proposition about decay of correlations.
\begin{proposition} \label{decayofcorr}
Generalizing the setting in Section~\ref{general}, let $(X_t)_{t \geq 0}$ be a Markov process on a topological state space $E$ with absorption at the boundary $\partial E$ with a quasi-stationary distribution $\nu$ and quasi-ergodic distribution $m$ satisfying (A1) and (A2). Then for any measurable and bounded functions $f, g: E \to \mathbb{R}$ and $ 0 < r < q < 1$, we have
$$ \lim_{t \to \infty} \mathbb{E}_{x} \big[ f(X_{qt}) g(X_{rt})\big |T > t \big] =  \int f\, \rmd m \int g \,\rmd m $$
uniformly for all $x \in E$.
\end{proposition}
\begin{proof}
Let $\eta:E\to \R^+_0$ be the bounded eigenfunction of $\mathcal L$ for the eigenvalue $\lambda_0$ such that $m(\rmd x) = \eta(x) \nu(\rmd x)$ (see Theorem~\ref{SDE_CCV}).

We consider first $f, g: E \to \mathbb{R}_0^+$, and let $x\in E$. Similarly to the proof of Theorem~\ref{SDE_CCV}, we fix $u >0$ and define the observable
$$ h_u(x) = \inf \big\{ e^{-\lambda_0 t} \mathbb{P}_{x}( T > t)/ \eta(x)\,:\,t\geq u \big\}\,.$$
Let $t$ be large enough such that $(q-r)t \geq u$ and $ (1-q)t \geq u$. We obtain with the Markov property
\begin{align*}
\mathbb{E}_{x} \big[ f(X_{qt}) g(X_{rt}) \big|T > t \big] &= \frac{\mathbb{E}_{x}\big[g(X_{rt}) f(X_{qt})\mathds 1_{\{T > t \}}\big]}{\mathbb{P}_{x}( T > t)} \\
&= \frac{\mathbb{E}_{x}\big[g(X_{rt}) f(X_{qt}) \mathds 1_{\{T > qt \}} \mathbb{P}_{X_{qt}}(T > (1-q)t ) \big]}{\mathbb{P}_{x}( T > t)}\\
&\geq \frac{ \mathbb{E}_{x}\big[ g(X_{rt}) \mathds 1_{\{ T > qt \}}  e^{-\lambda_0 (q-1)t}f(X_{qt}) h_u(X_{qt}) \eta(X_{qt}) \big ]}{ \mathbb{P}_{x}( T > t)}\,.
\end{align*}
Let us denote $\rho(x) = f(x) h_u(x) \eta(x)$. Recall that $\eta$ is bounded, and the limit \eqref{etadef} in the definition of $\eta$ is uniform in $x$. Hence, there exists a constant $C > 0$ such that for all $t \geq u$ and $x \in E$,
\begin{equation} \label{bound_hu}
\left| \rho(x) \right| = \left| f(x) h_u(x) \eta(x) \right| \leq  \left| f(x) e^{-\lambda_0 t} \mathbb{P}_{x}( T > t) \right| \leq C \|f\|_{\infty} \|\eta\|_{\infty}\,.
\end{equation}
The same holds for $\tilde{\rho}(x) = g(x) h_u(x) \eta(x)$, and hence, the functions $\rho$ and $\tilde \rho$ are bounded and obviously measurable.
Similarly to the estimate above, using the Markov property, we obtain
\begin{align}
& \mathbb{E}_{x}\big [ f(X_{qt}) g(X_{rt}) \big|T > t \big] \geq e^{-\lambda_0 (q-1)t} \frac{ \mathbb{E}_{x}\big[ g(X_{rt}) \mathds 1_{\{T > rt \}} \mathbb{E}_{X_{rt}} [ \mathds 1_{\{ T > (q-r)t \}}  \rho(X_{(q-r)t})]  \big]}{ \mathbb{P}_{x}( T > t)} \nonumber\\
\vspace{0.5cm}
&\geq e^{-\lambda_0 (q-1)t}  \mathbb{E}_{x}\left[\frac{ g(X_{rt}) h_u(X_{rt}) \eta (X_{rt}) \mathds 1_{\{ T > rt \}} \mathbb{E}_{X_{rt}} [ \mathds 1_{\{ T > (q-r)t\}}  \rho(X_{(q-r)t})]  }{ \mathbb{P}_{x}( T > t) \mathbb{P}_{X_{rt}} (T > t) e^{-\lambda_0 t} }\right] \nonumber\\
\vspace{0.5cm}
&=   \mathbb{E}_{x}\left[\frac{ e^{-\lambda_0 rt} \tilde{\rho}(X_{rt}) \mathds 1_{\{ T > rt \}}}{\mathbb{P}_{x}( T > t) e^{-\lambda_0 t}}\cdot \frac{\mathds 1_{\{ T > rt \}}e^{-\lambda_0 (q-r)t} \mathbb{E}_{X_{rt}} [ \mathds 1_{\{ T > (q-r)t \}}  \rho(X_{(q-r)t})]  }{  \mathbb{P}_{X_{rt}} (T > t) e^{-\lambda_0 t} } \right]\,.\label{estimatemarkov}
\end{align}
For the second factor in the expectation in \eqref{estimatemarkov}, we obtain
\begin{eqnarray*}
 && \!\!\!\left| \mathds 1_{\{ T > rt \}} \frac{ e^{-\lambda_0 (q-r)t} \mathbb{E}_{X_{rt}} [ \mathds 1_{\{ T > (q-r)t \}}  \rho(X_{(q-r)t})]  }{  \mathbb{P}_{X_{rt}} (T > t) e^{-\lambda_0 t} } -  \int \rho \, \rmd \nu \right| \\
&= &\!\!\! \left| \mathds 1_{\{ T > rt \}} \frac{e^{-\lambda_0 (q-r)t} \mathbb{P}_{X_{rt}}( T > (q-r)t)}{ e^{-\lambda_0 t} \mathbb{P}_{X_{rt}}( T > t)}\mathbb{E}_{X_{rt}}[\rho(X_{(q-r)t})| T > (q-r)t] - \int \rho \, \rmd \nu \right|\\
 &\stackrel{\eqref{qsdexpconv}}{\leq} &\!\!C e^{ - \gamma t} + \| \rho \|_{\infty} \left| \mathds 1_{\{ T > rt \}} \frac{e^{-\lambda_0 (q-r)t} \mathbb{P}_{X_{rt}}( T > (q-r)t)}{ e^{-\lambda_0 t} \mathbb{P}_{X_{rt}}( T > t)}  - 1 \right| \\
&\leq & \!\!\!C e^{- \gamma t} + \| \rho \|_{\infty} \sup_{x \in E} \left| \frac{e^{-\lambda_0 (q-r)t} \mathbb{P}_{x}( T > (q-r)t)}{ e^{-\lambda_0 t} \mathbb{P}_{x}( T > t)}  - \frac{\eta(x)}{\eta(x)} \right|\,,
\end{eqnarray*}
which converges to $0$ as $t \to \infty$, since the convergence in \eqref{etadef} is uniform in $x$.
Hence, the second factor in the expectation in \eqref{estimatemarkov} converges uniformly to
$$ \int \rho \, \rmd \nu =  \int f h_u \, \rmd m \,.$$
We observe that
\begin{align*}
& \liminf_{t \to \infty} \mathbb{E}_{x} \big[ f(X_{qt}) g(X_{rt}) \big|T > t \big] \\
&\geq  \lim_{t \to \infty}  \mathbb{E}_{x}\left[\frac{ e^{-\lambda_0 rt} \tilde{\rho}(X_{rt}) \mathds 1_{\{ T > rt \}}}{\mathbb{P}_{x}( T > t) e^{-\lambda_0 t}} \frac{e^{-\lambda_0 (q-r)t} \mathbb{E}_{X_{rt}} [ \mathds 1_{\{ T > (q-r)t \}}  \rho(X_{(q-r)t})]  }{  \mathbb{P}_{X_{rt}} (T > t) e^{-\lambda_0 t} } \right] \\
&= \lim_{t \to \infty}  \frac{\mathbb{E}_{x}\left[ e^{-\lambda_0 rt} \tilde{\rho}(X_{rt}) \mathds 1_{\{ T > rt \}} \right]}{\mathbb{P}_{x}(T > t) e^{-\lambda_0 t}}  \int f h_u \, \rmd m\\
&= \int g h_u \, \rmd m \int f h_u \, \rmd m\,.
\end{align*}
Since $h_u$ is uniformly bounded and $h_u(x) \to 1$ as $u\to\infty$, we have by the dominated convergence theorem that
$$ \liminf_{t \to \infty} \mathbb{E}_{x} \big[ f(X_{qt}) g(X_{rt}) \big|T > t \big] \geq  \int f \,\rmd m \int g \,\rmd m \,.$$
Replacing $f(X_{qt}) g(X_{rt})$ by
$$(\| f \|_{\infty} - f(X_{qt}))(\| g \|_{\infty} + g(X_{rt}))\quad \text{and} \quad (\| f \|_{\infty} + f(X_{qt}))(\| g \|_{\infty} - g(X_{rt}))\,,$$
we can see directly that
\begin{align*}
&2 \| f \|_{\infty} \| g \|_{\infty} - \limsup_{t \to \infty} \mathbb{E}_{x} \big[ 2 f(X_{qt}) g(X_{rt}))\big |T > t \big] \\
& \geq \liminf_{t \to \infty} \mathbb{E}_{x} \big[  (\| f \|_{\infty} - f(X_{qt}))(\| g \|_{\infty} + g(X_{rt})) \big|T > t \big] \\
&\quad+ \liminf_{t \to \infty} \mathbb{E}_{x} \big[  (\| f \|_{\infty} + f(X_{qt}))(\| g \|_{\infty} - g(X_{rt})) \big|T > t \big] \\
& \geq 2 \| f \|_{\infty} \| g \|_{\infty} - 2\int f \,\rmd m \int g \,\rmd m\,.
\end{align*}
Therefore, we deduce that
\begin{equation*}
\limsup_{t \to \infty} \mathbb{E}_{x} \big[ f(X_{qt}) g(X_{rt}) \big|T > t \big] \leq  \int f \,\rmd m \int g \rmd m \,.
\end{equation*}
We have shown that the claim holds for non-negative measurable and bounded functions $f,g$. We can extend the result to arbitrary measurable and bounded functions $f,g$ by replacing $fg$ with $(f_+ - f_-)(g_+ - g_-)$. Uniformity of the convergence follows for the same reasons as in Theorem~\ref{SDE_CCV}.
\end{proof}

We provide the following detailed version of Theorem B, which equips the limit of expected values $\lambda$, as given in Theorem~\ref{fk_killed} and Proposition~\ref{Lyaponedim}, with the strongest possible dynamical meaning in the setting of killed processes.

\begin{theorem}[Convergence in conditional probability]  \label{coninprob_md}
Consider the stochastic differential equation \eqref{MultidimSDE} corresponding to the Markov process $(X_t)_{t \geq 0}$ that is killed at $\partial E$ such that the conditioned Lyapunov exponent $\lambda$ exists.
Then for all $\epsilon > 0$, we have
\begin{equation} \label{weak law_md}
\lim_{t  \to \infty} \mathbb{P} \Big( \big| \lambda_v(t,\cdot,x) - \lambda \big| \geq \epsilon \Big| \tilde T(\cdot,x) > t \Big) = 0
\end{equation}
uniformly for all $x \in E$ and $v \in \mathbb S^{d-1}$.
This means that the finite-time Lyapunov exponents of the surviving trajectories converge to the Lyapunov exponent $\lambda$ in probability.
\end{theorem}
\begin{proof}
Recall from above that
\begin{align*}
\lambda &= \lim_{t \to \infty} \mathbb{E}_{x, s_0} [ \lambda_v(t, \cdot, x) | T > t ] = \lim_{t \to \infty} \mathbb{E} \big[ \lambda_v(t, \cdot, x) \big| \tilde T(\cdot,x) > t \big] \\
&= \int_{\mathbb S^{d-1} \times E}  \langle s, \rmD f(x) s \rangle \ \tilde{m}(\rmd s, \rmd x)\,.
\end{align*}
In the following, we will write $g(x,s) = \langle s, \rmD f(x) s \rangle $ and
$\tilde m(h) := \int h \, \rmd \tilde m$
for any bounded and measurable function $h$. Note that in dimension one, we have $g(x,s) = f'(x)$ and $\tilde m = m$.
We observe that
\begin{align}
&\mathbb{P} \big( | \lambda_v(t,\cdot,x) - \lambda| \geq \epsilon \big| \tilde T(\cdot,x) > t \big) \nonumber\\ & \leq
\mathbb{P} \Big( \big| \lambda_v(t,\cdot,x) - \mathbb{E} \big[ \lambda_v(t, \cdot, x) \big| \tilde T(\cdot,x) > t \big] \big| \geq \epsilon \Big| \tilde T(\cdot,x) > t \Big)\label{rel1} \\
 &\quad+ \mathbb{P} \left( \big| \lambda - \mathbb{E} \big[ \lambda_v(t, \cdot, x) \big| \tilde T(\cdot,x) > t \big] \big| \geq \epsilon \Big| \tilde T(\cdot,x) > t \right)\label{rel2}\,.
\end{align}
The term in \eqref{rel1} converges to zero as $t\to\infty$ by definition of $\lambda$. The term in \eqref{rel2} can be estimated by Chebyshev's inequality:
$$ \mathbb{P} \Big( \big| \lambda_v(t,\cdot,x) - \mathbb{E}\big[ \lambda_v(t, \cdot, x) \big| \tilde T(\cdot,x) > t \big] \big| \geq \epsilon \Big| \tilde T(\cdot,x) > t \Big) \leq \frac{\Var\big[\lambda_v(t, \cdot,x)\big| \tilde T(\cdot,x) > t\big] }{\epsilon^2}\,.$$
This means that, in order to prove the claim, we  need to show that
$$ \lim_{t \to \infty} \mathbb{E} \big[ \lambda_v(t, \cdot, x)^2 \big| \tilde T(\cdot,x) > t \big] = \lim_{t \to \infty} \Big(\mathbb{E} \big[ \lambda_v(t, \cdot, x)\big | \tilde T(\cdot,x) > t \big]\Big)^2\,,$$
where
$$\lim_{t \to \infty} \Big(\mathbb{E} \big[ \lambda_v(t, \cdot, x)\big | \tilde T(\cdot,x) > t \big]\Big)^2 = \lambda^2 = \tilde m(g)^2 \,.$$
We obtain with Fubini's Theorem that
\begin{align*}
&\lim_{t \to \infty} \mathbb{E} \big[ \lambda_v(t, \cdot, x)^2 \big| \tilde T(\cdot,x) > t \big]\\
 &=
\lim_{t \to \infty} \mathbb{E}_x \left[ \left(\frac{1}{t} \int_{0}^t g(X_{\tau}, s_{\tau})\, \rmd \tau \right)^2  \bigg| T > t \right]
= \lim_{t \to \infty} \mathbb{E}_x \left[ \left( \int_{0}^1 g(X_{qt},s_{qt}) \, \rmd q \right)^2  \bigg| T > t \right] \\
&= \lim_{t \to \infty}  \int_{0}^1 \int_{0}^1 \mathbb{E}_x \big[g(X_{qt},s_{qt}) g(X_{rt},s_{rt})\big| T > t \big]\rmd q \, \rmd r \\
&= \lim_{t \to \infty} \bigg( \int_{0}^1 \int_{0}^q \mathbb{E}_x \big[g(X_{qt},s_{qt}) g(X_{rt},s_{rt}) \big| T > t \big]\rmd r \, \rmd q
+ \int_{0}^1 \int_{0}^r \mathbb{E}_x \big[g(X_{qt},s_{qt}) g(X_{rt},s_{rt}) \big| T > t \big]\rmd q \, \rmd r  \bigg)\,.
\end{align*}
It follows immediately from Proposition~\ref{decayofcorr} that for $0 < r < q <1$ (and $0 < q < r <1$),
$$ \lim_{t \to \infty}  \mathbb{E}_x \big[g(X_{qt},s_{qt}) g(X_{rt},s_{rt}) \big| T > t \big] = \lim_{t \to \infty}  \mathbb{E}_{x, s_0} \big[g(X_{qt},s_{qt}) g(X_{rt},s_{rt}) \big| T > t \big]= \tilde m(g)^2\,,$$
where the convergence is uniform over the initial values.
Hence, by using dominated convergence, we conclude that
$$ \lim_{t \to \infty} \mathbb{E} \big[ \lambda_v(t, \cdot, x)^2 \big| \tilde T(\cdot,x) > t \big] = \int_{0}^1 \int_{0}^q \tilde m(g)^2 \, \rmd r \, \rmd q +  \int_{0}^1 \int_{0}^r \tilde m(g)^2 \, \rmd q \, \rmd r = \tilde m(g)^2 \,,$$
which finishes the proof of this theorem.
%
\end{proof}
\subsection{Synchronization} \label{locsynchnonlin}

We consider negative Lyapunov exponents $\lambda$ in this subsection, and we show Theorem C from the Introduction, addressing that surviving trajectories starting close enough to each other are synchronizing.
\begin{theorem}[Local synchronization theorem] \label{LocalSM}
Consider the stochastic differential equation \eqref{MultidimSDE} with conditioned Lyapunov exponent $\lambda <0$.
Then for all $\lambda_{\epsilon} \in (\lambda, 0)$, $x \in E$ and $\rho\in( 0, 1)$, there exists an $\alpha_x > 0$ such that
$$
\lim_{t \to \infty}
\mathbb{P} \left( \tfrac{1}{t} \ln \left\|
\varphi(t,\cdot,x)-\varphi(t,\cdot,y)\right\| \leq \lambda_{\epsilon} \mbox{ for all } y \in B_{\alpha_x}(x)
\Big|  \tilde T(\cdot,x) > t \right) > 1 - \rho\,. $$
\end{theorem}

\begin{proof}
This proof follows \cite[Theorem~5.1]{r79} and is adapted to the setting of conditioned random processes. It is divided in two steps.

\noindent\emph{Step 1.} We show that for all $\lambda_{\epsilon} \in (\lambda, 0)$, $x \in E$ and $\rho\in( 0, 1)$, there are $\alpha_x > 0$, $\beta \in(0, 1)$, and sets $\Omega_x^n \subset \Omega$ with $\mathbb{P}_x\big(\Omega_x^n \big|T>n \big) > 1 - \rho$ for all $n \in \mathbb{N}$ such that for all $n \in \mathbb{N}$, $\omega \in \Omega_x^n$  and $y \in B_{\alpha_x}(x)$, we have
\begin{equation}\label{rel3}
  \| \varphi(n,\omega,x) - \varphi(n,\omega,y) \| \leq \beta e^{\lambda_{\epsilon}n}\,.
\end{equation}

Let $\rho\in(0, 1)$ be fixed for the following. For $x \in E$, there are $0 < \gamma < \dist(x, \partial E)$ and $E_x \subset E$ with $x \in E_x$ and $\dist(E_x,E) \geq \gamma$ such that
\begin{equation} \label{rhohalf}
\mathbb{P}_x\big(X_n \in E_x |T>n \big) > 1 - \frac{\rho}{2} \fa n \in \mathbb N\,.
\end{equation}
Let $\Omega_x^n \subset \big\{ \omega \in \Omega : X_n(\omega) \in E_x, \ \tilde T(\omega,x) > n \big \}$ be a set with $\mathbb{P}_x\big(\Omega_x^n \big|T>n \big) > 1 - \rho$ (note that the construction of this set is given in \eqref{def_omega_n} below). Furthermore, we define for any $x \in E$
$$ U_x :=  \big\{ y \in \mathbb{R}^d\,: \ x + y \in \overline{E}\big\}\,.$$
For fixed $(\omega,x)  \in \Omega \times E$, we define for $y\in U_x$
$$ Z_n((\omega,x),y) := \varphi(n, \omega, y + x) - \varphi(n, \omega, x)\,.$$
Note that in particular $Z_n((\omega,x),0) = 0$ for all $n\in\mathbb N$. Define further
$$ F_{(\omega,x)}(y) := Z_1((\omega,x),y)$$
and write
$$ F_{(\omega,x)}^n = F_{\Theta^{n-1}(\omega,x)} \circ \dots \circ F_{(\omega,x)}\,.$$
In addition, we define
$$ L(\omega,x) = \rmD F_{(\omega,x)}(0) = \rmD \varphi(1, \omega, x)\,, $$
and for all $n \geq 1$,
$$ L_n(\omega,x) = L(\Theta_{n-1}(\omega,x))\,.$$
Similarly to \cite{r79}, let $ 0 < \eta = - \frac{1}{2}\lambda_{\epsilon}$.
Since \eqref{MultidimSDE} is $C^1$ on the bounded domain $E$, we have
$$ G := \sup_{(\omega,x)\in\Omega\times E} \|F_{(\omega,x)}\|_{C^1} < \infty\,.$$
Let $ \delta > 0$ be given. Choose $ 0 < \beta  < 1$ such that $G \beta e^{\eta} < \delta$. Take further $\kappa > 1$ such that $\kappa \beta \leq 1$ and $G \kappa \beta e^{\eta} \leq \delta$, and recall that
$$ \Omega_x^k \subset \big\{ \omega \in \Omega :  \tilde T(\omega,x) > k  \big\}\fa k\in\mathbb N\,.$$
Define
$$S^k(\beta) = \big\{ y \in U_x: \|F_{(\omega,x)}^n(y)\|\leq \beta e^{ \lambda_{\epsilon} n} \mbox{ for all } 0 \leq n \leq k \text{ and } \omega \in \Omega_x^k  \big\} \fa k \in\mathbb N\,.$$
We assume $\kappa \beta  \leq \gamma$ to guarantee that $S^k(\kappa \beta)$ and $S^k(\beta)$ are non-empty and have disjoint boundaries.
For $k\in \mathbb N$ and $y \in S^k(\kappa \beta)$, we define for all $ 1 \leq n \leq k$
\begin{equation*}
 L'_n(\omega,x) = \int_{0}^1 \rmD F_{\Theta_{n-1}(\omega,x)}\big(tF_{(\omega,x)}^{n-1}(y)\big)\,\rmd t \,.
\end{equation*}
Observe that this choice yields for $ 1 \leq n \leq k$
\begin{equation*}
L'^n(\omega,x) y =  L'_n(\omega,x) \cdots   L'_1(\omega,x)y = F_{(\omega,x)}^n(y)\,.
\end{equation*}
We deduce that for any $k\in\mathbb N$ and $y \in S^k(\kappa \beta)$
\begin{align*}
 \sup_{\omega \in \Omega_x^k} \sup_{n \leq k} \|  L'_n(\omega,x) -  L_n(\omega,x) \| e^{ \eta n} &\leq \sup_{\omega \in \Omega_x^k} \sup_{n \leq k} \| \rmD F_{\theta_{n-1} (\omega,x)} \| \kappa \beta \exp{\big(n( \eta + \lambda_{\epsilon}) - \lambda_{\epsilon} \big)}\\
&\leq G \kappa \beta e^{\eta} < \delta\,.
\end{align*}
\emph{Claim}. There exists a $K_{\epsilon} > 1$ such that for all $k\in\mathbb N$, $\omega \in \Omega_x^k$, $y \in S^k(\kappa \beta)$ and $ 1 \leq n \leq k$, we have
$$ \|  L'^n(\omega,x) y\| \leq K_{\epsilon} e^{n \lambda_\epsilon} \|y\|\,.$$

Let us first assume that the claim is true and define $d_x = \frac{1}{2} \dist(x, \partial E)$.
Set $\alpha_x := \min{(d_x, \beta/K_{\epsilon}) } < \beta$. From the claim, we observe that for all $k\in\mathbb N$, $y \in \overline{B_{\alpha_x}(0)} \cap S^k(\kappa \beta)$ and $1 \leq n \leq k,$
$$ \|  F_{(\omega,x)}^n (y)\| \leq K_{\epsilon} e^{n \lambda_\epsilon} \alpha_x \leq \beta e^{n \lambda_\epsilon}\,, $$
and therefore,
$$ D_k(\alpha_x) := \overline{B_{\alpha_x}(0)} \cap S^k(\beta) = \overline{B_{\alpha_x}(0)} \cap S^k(\kappa \beta)\,.$$
Since the boundaries of $S^k(\beta)$ and $S^k(\kappa \beta)$ are disjoint, this implies that $\overline{B_{\alpha_x}(0)} = D_k(\alpha_x)$  for all $k > 0$ and \eqref{rel3} follows.


We now prove the above claim. Fix $(\omega,x)\in \Omega\times E$. Set $t^{(n)} := \|L_n(\omega,x) \|$.
It is enough to consider the one-dimensional case, since the convergence in \eqref{weak law_md} is assumed to be uniform over all $v  \in \mathbb S^{d-1}$ in the $d$-dimensional scenario (see \cite{r79} for higher-dimensional details in the unconditioned case).

We write $u^{(n)}:=L'^n(\omega,x)$ and see that $\sup_{n\in\mathbb N} |  L'_n(\omega,x) -  L_n(\omega,x)| e^{\eta n} < \delta$ implies
\begin{equation} \label{growthinequality}
\big| u^{(n)}  \big| \leq t^{(n)} \big| u^{(n-1)}  \big|  + \delta e^{-n\eta} \big| u^{(n-1)}  \big| \fa n \in\mathbb N\,.
\end{equation}
Since the finite-time Lyapunov exponents are bounded below according to the assumptions of the model, there is a $ C > 0$, independent from $(\omega,x)$, such that
$$ \frac{1}{C} e^{-N\eta} \leq t^{(N)} \fa N \in \mathbb{N}\,.$$
We set $U^{(0)} := \left| u^{(0)}\right|$ and
$$ U^{(N)} := \left( \prod_{n=n_0 +1}^N t^{(n)} \right) \left(\prod_{n=n_0 +1}^N (1 + C\delta e^{-n\eta})\right) U^{(n_0)} \fa N \in \mathbb N\,.$$
We observe with \eqref{growthinequality} that $U^{(N)} \geq \left| u^{(N)}\right|$ for all $N \in \mathbb N_0$.
Now we set $\delta := \frac{1}{C} \prod_{n=1}^\infty (1 - e^{-n\eta})^2$ and
$$C' := \frac{\prod_{n=1}^{\infty} (1 + C\delta e^{-n\eta})}{\prod_{n=1}^{\infty} (1 -  e^{-n\eta})} \leq \prod_{n=1}^{\infty} (1 - e^{-n\eta})^{-2} = \frac{1}{C \delta}\,.$$
Note that $\delta$ and $C'$ do not depend on $(\omega,x)$.
It is easy to infer similarly to \cite{r79} that
\begin{equation} \label{growthabove}
\big| u^{(N)} \big| \leq C' \left(\prod_{n=1}^N t^{(n)} \right) \left(\prod_{n=1}^N (1 - e^{-n\eta})\right)\big| u^{(0)} \big|\fa N\in \mathbb N\,.
\end{equation}
Observe that the finite Lyapunov exponents satisfy
$$ \lambda(N, \omega,x) = \frac{1}{N} \ln \prod_{n=1}^N t^{(n)}\,.$$
Let $\epsilon := \lambda_{\epsilon} - \lambda > 0$.
By Theorem~\ref{coninprob_md}, there exists an $N^* > 0$ such that for all $N \geq N^*$, we have
\begin{equation} \label{rhohalf2}
 \mathbb{P} \left(  \lambda(N, \cdot, x)  < \lambda + \frac{\epsilon}{2} \Big| \tilde T(\cdot,x) > N \right) > 1- \frac{\rho}{2}\,.
 \end{equation}
Define the measurable sets
\begin{equation} \label{def_omega_n}
\Omega_x^N =
\begin{cases}
\big\{ \omega\in \Omega \,:\,  X_n(\omega) \in E_x, \ \tilde T(\omega,x) > N , \ \lambda(N, \omega, x)  < \lambda + \epsilon/2\big\} \ &\text{ if } N \geq N^*\,, \\
\big \{ \omega \in \Omega \,:\,  X_n(\omega) \in E_x, \ \tilde T(\omega,x) > N  \big\} \ &\text{ if } N < N^* \,.
\end{cases}
\end{equation}
Hence, we can deduce from~\eqref{rhohalf} and~\eqref{rhohalf2} that $\mathbb{P}_x\big(\Omega_x^n \big| T > n\big)> 1 - \rho$ for all $n \in \mathbb{N}$.
Recall that $\delta, C, C'$ do not depend on $ (\omega,x)$.
We conclude from \eqref{growthabove} that there is $C'' > 0$  such that for all $x \in E$, $N \geq N^*$ and $\omega \in \Omega_x^N$,
$$ \frac{1}{N} \ln | L'^N(\omega,x) | \leq \frac{C''}{N} + \frac{1}{N} \ln \prod_{n=1}^N t^{(n)} = \frac{ C''}{N} + \lambda(N, \omega,x)\,. $$
Since $\overline{E}$ is compact and $f$ is $C^1$, we have that
$$ s^*:= \sup \big\{\lambda(N,\omega,x): N\le N^*, (\omega,x)\in\Omega\times E \mbox{ such that }\tilde T(\omega,x)\ge N^*\big\} < \infty\,.$$
We define
$$ K_{\epsilon} := \max \big\{e^{C''}, e^{-\lambda_{\epsilon}N^*} s^*\big\}\,.$$
Then the claim follows, i.e.~for all  $N \in \mathbb N$ and $\omega \in \Omega_x^N$, we have
$$ \big| L'^N(\omega,x) \big| \leq K_{\epsilon} e^{ \lambda_{\epsilon} N} \,.$$

\noindent\emph{Step 2.} We finally prove the statement of the theorem. Due to Step 1 and $\beta < 0$, we have
\begin{align*}
&\mathbb{P} \left( \left\|
\varphi(n,\cdot,x)-\varphi(n,\cdot,y)\right\| \leq  e^{ \lambda_{\epsilon} n} \mbox{ for all } y \in B_{\alpha_x}(x) \Big| \tilde T(\cdot,x) > n \right) \\
&\geq \mathbb{P} \big( \Omega_x^n \big| \tilde T(\cdot,x) > n \big)  > 1 - \rho \fa n \in \mathbb{N}\,.
\end{align*}
Hence, we obtain
$$ \lim_{n\to\infty} \mathbb{P} \left( \tfrac{1}{n} \ln \left\|
\varphi(n,\omega,x)-\varphi(,\omega,y)\right\| \leq \lambda_{\epsilon} \mbox{ for all }y \in B_{\alpha_x}(x) \Big| \tilde T(\cdot,x) > n \right) > 1 - \rho \,.$$
Since $\overline{E}$ is compact and $f$ is $C^1$, in this statement, $n\in\mathbb N$ can be replaced by $t\ge 1$, which finishes the proof of this theorem.
\end{proof}

To apply the above theorem, one would need to check that the conditioned Lyapunov exponent $\lambda$ is negative. Recall that Proposition~\ref{bounds} gave bounds from above (and below), which could help to establish the assumption on negativity of $\lambda$. On the other hand, in dimension one, Proposition~\ref{Lyaponedim} provides the explicit formula \eqref{lambda_onedim} for $\lambda$, which depends on the eigenfunction  $\psi$ of $\mathcal L$ for the eigenvalue $\lambda_0$, i.e.
\begin{equation} \label{sturm_liouville}
\frac{1}{2} \sigma^2 \psi''(x) + f(x) \psi'(x) = \lambda_0 \psi(x)\,.
\end{equation}
Using integration by parts, we observe that
\begin{align*}
\lambda=\int_E f'(x) \psi^2(x) e^{\gamma(x)}\rmd x &= - \int_E f(x)\left( \frac{2}{\sigma^2} f(x) \psi^2(x) e^{\gamma(x)}  +  2 \psi(x) \psi'(x) e^{\gamma(x)} \right) \rmd x \\
&= \underbrace{- \frac{2}{\sigma^2} \int_E f^2(x)  m(\rmd x)}_{\leq 0}  - \underbrace{2 \int_E f(x) \psi'(x)  \psi(x )e^{\gamma(x)} \, \rmd x }_{ \text{ sign unclear}}\,.
\end{align*}
Note that the second term vanishes in case $E = \mathbb{R}$ since $\psi$ is a constant function in this case and $m$ is the stationary distribution. That is how one observes negativity of the Lyapunov exponent in the classical unconditioned setting. In our context, the sign of the second term depends on the product $f(x) \psi'(x)$ which makes a direct a priori estimate impossible. However, using \eqref{sturm_liouville}, we get
\begin{align*}
\lambda=\int_E f'(x) \psi^2(x) e^{\gamma(x)}\rmd x &= - \int_E f(x)\left( \frac{2}{\sigma^2} f(x) \psi^2(x) e^{\gamma(x)}  +  2 \psi(x) \psi'(x) e^{\gamma(x)} \right) \rmd x \\
&= \underbrace{- \frac{2}{\sigma^2} \int_E f^2(x) \, m(\rmd x)}_{\leq 0} \underbrace{- 2 \lambda_0}_{ > 0} + \underbrace{\sigma^2 \int_E  \psi''(x)  \psi(x )e^{\gamma(x)}\,  \rmd x }_{ \text{ sign unclear}}\,.
\end{align*}
We remain with a term whose sign is unclear, this time depending on $\psi''(x)$. It is not possible to obtain general statements about the sign of $\lambda$, but the above analysis may help to establish negativity of $\lambda$ in particular cases or numerically.

\section{The conditioned dichotomy spectrum} \label{expdich_sec}

It has been shown recently in \cite{cdlr16} that, in addition to Lyapunov exponents, the dichotomy spectrum is useful for the study of bifurcations in random dynamical systems. In this section, we introduce the conditioned dichotomy spectrum and discuss basic properties.

Although we remain in the context of killed random dynamical systems, we do not need quasi-stationary and quasi-ergodic distributions for the analysis of the dichotomy spectrum. Thus, we formulate everything for more general random dynamical systems that include those generated by the stochastic differential equation~\eqref{MultidimSDE} as a special case.

Consider a bounded open set $E\subset \R^d$, and let $\Theta = (\theta: \R\times \Omega\to\Omega, \varphi: D\subset \R \times \Omega\times E\to \R^d)$ be a $C^1$ random dynamical system such that for all $x\in E$ and almost all $\omega \in \Omega$, we have that $D_{\omega,x}:= \{t\in\R: (t,\omega,x) \in D\}$ is an interval containing $0$ in its interior.

Consider the linear random dynamical system $(\Theta, \Phi)$, where $\Phi(t, \omega, x) \in \R^{d\times d}$ is given by $\Phi(t, \omega, x) = \frac{\partial \varphi}{\partial x} (t, \omega, x)$.
As before, we consider killing at the boundary $\partial E$. We define the hitting times of the boundary in forward and backward time $T^+: \Omega \times E \to \mathbb{R}^+$ and $T^-: \Omega  \times E \to \mathbb{R}^-$ by
\begin{equation} \label{stoppingtimes}
T^+(\omega,x) := \inf \big\{ t > 0 : \varphi(t, \omega,x) \in \partial E \big\} \quad \mbox{ and } \quad  T^-(\omega,x) := \sup \big\{ t < 0: \varphi(t, \omega,x) \in \partial E \big\}\,.
\end{equation}
We assume that $T^+(\cdot, x) < \infty$ and $T^-(\cdot, x) > - \infty$ almost surely for all $x \in E$ and that for all $x\in E$, we have
$$\mathbb{P} \big(T^+(\cdot, x) > t\big) > 0 \text{ for all } t \geq 0 \quad \text{ and } \quad \mathbb{P} \big(T^-(\cdot, x) < t\big) > 0 \text{ for all } t \leq 0\,.$$
Note that $\big(T^-(\omega,x),T^+(\omega,x)\big)\subset D_{\omega,x}$.
To avoid ambiguities, we will write $(\Theta, \Phi, T^+, T^-)$ for the whole system.

We assume a boundedness condition for $\Phi$ on compact time intervals, which is automatically fulfilled if $\Phi$ is the linearization of \eqref{MultidimSDE}.
\begin{assumption}\label{assds}
  For all $t^*>0$, there exists $M>0$ such that for almost all $\omega\in\Omega$, we have
  \begin{displaymath}
    \|\Phi(t,\omega,x)\|\le M \fa x\in E \mbox{ and } t\in \big[\max\{-t^*,T^-(\omega,x)\},\min\{t^*,T^+(\omega,x)\}\big]\,.
  \end{displaymath}
\end{assumption}

First, we define invariant projectors for this setting.
\begin{definition}[Invariant projector] \label{projector_killed}
  Consider the linear random dynamical system $(\Theta, \Phi)$ with stopping times $T^+, T^-$ as given in~\eqref{stoppingtimes}.
  An \emph{invariant projector} $P$ for $(\Theta, \Phi, T^+, T^-)$ is a measurable function $P: \Omega \times E \to \mathbb{R}^{d \times d}$ such that for almost all $\omega \in \Omega$ and all $x \in E$, we have
  \begin{enumerate}[(i)]
  \item $P(\omega,x) = P(\omega,x)^2$,
  \item $P(\Theta_t (\omega,x)) \Phi(t, \omega,x) = \Phi(t, \omega,x)P(\omega,x) \text{ for all } T^-(\omega,x) < t < T^+(\omega,x)$,
  \item and there exists an $r\in\{0,\dots, d\}$ such that for almost all $\omega\in \Omega$ and all $x\in E$, $\rk(P(\omega,x))$ is equal to~$r$.
  \end{enumerate}
\end{definition}

Note that in \cite[Proposition 2.1]{cdlr16}, (iii) follows from measurability of the invariant projector if the random dynamical system is not absorbed at the boundary $\partial E$.

We denote the null space and range of $P$ by
\begin{align*}
\mathcal{N}(P) &= \big\{ (\omega,x,v) \in \Omega \times E \times \mathbb{R}^d\,:\, P(\omega,x)v=0\big\}\,, \\
\mathcal{R}(P) &= \big\{ (\omega,x,v) \in \Omega \times E \times \mathbb{R}^d\,:\, P(\omega,x)w=v \text{ for some } w \in \mathbb{R}^d\big\}\,.
\end{align*}

Due to property (iii) of Definition~\ref{projector_killed}, we define
$$ \rk P := \dim \mathcal{R}(P) := r \quad \mbox{ and } \quad \dim \mathcal{N}(P) := d-r\,.$$

We now give the following definition of an exponential dichotomy for the system with absorption at the boundary.
\begin{definition}[Exponential dichotomy]
Consider the linear random dynamical system  $(\Theta, \Phi)$ with stopping times $T^+, T^-$, and let $\gamma \in \mathbb{R}$ and $P_{\gamma}$ be an invariant projector for $(\Theta, \Phi, T^+, T^-)$.
Then $(\Theta, \Phi, T^+, T^-)$ is said to admit an \emph{exponential dichotomy} with growth rate $\gamma$, constants $\alpha > 0$, $K \geq 1$ and projector $P_{\gamma}$ if for almost all $\omega \in \Omega$ and all $x \in E$
\begin{align*}
\| \Phi(t, \omega,x)P_{\gamma}(\omega,x) \| &\leq K e^{(\gamma - \alpha)t} \fa 0 \leq t < T^+(\omega,x)\,, \\
\| \Phi(t, \omega,x)(\Id - P_{\gamma}(\omega,x)) \| &\leq K e^{(\gamma + \alpha)t} \fa 0 \geq t > T^-(\omega,x)\,.
\end{align*}
\end{definition}

We say that $(\Theta, \Phi, T^+, T^-)$ admits an exponential dichotomy with growth rate $\infty$ if there exists a $\gamma \in \mathbb{R}$ such that $(\Theta, \Phi, T^+, T^-)$ admits an exponential dichotomy with growth rate $\gamma$ and projector $P_{\gamma} = \Id$.
Analogously, we say that $(\Theta, \Phi, T^+, T^-)$ admits an exponential dichotomy with growth rate $-\infty$ if there exists a $\gamma \in \mathbb{R}$ such that $(\Theta, \Phi, T^+, T^-)$ admits an exponential dichotomy with growth rate $\gamma$ and projector $P_{\gamma} = 0$. We write $ \overline{\mathbb{R}} = \mathbb{R} \cup \{-\infty, \infty\}$.

Analogously to \cite[Lemma~2.4]{cdlr16}, directly from the definitions, the following observation follows.
\begin{lemma} \label{P_infty}
Suppose that $(\Theta, \Phi, T^+, T^-)$ admits an exponential dichotomy with growth rate $\gamma$ and projector $P_{\gamma}$. Then the following statements are satisfied:
\begin{enumerate}[(i)]
\item If $P_{\gamma} = \Id $ almost surely, then $(\Theta, \Phi, T^+, T^-)$ admits an exponential dichotomy with growth rate $\zeta$ and invariant projector $P_{\zeta} = \Id$ for all $ \zeta > \gamma$.
\item If $P_{\gamma} = 0 $ almost surely, then $(\Theta, \Phi, T^+, T^-)$ admits an exponential dichotomy with growth rate $\zeta$ and invariant projector $P_{\zeta} = 0$ for all $ \zeta < \gamma$.

\end{enumerate}

\end{lemma}

Finally, we define the dichotomy spectrum in our setting.
\begin{definition}[Dichotomy spectrum] \label{dich_spectrum_killed}
Consider the linear random dynamical system $(\Theta, \Phi)$ with stopping times $T^+, T^-$. Then its \emph{dichotomy spectrum} is defined by
$$  \Sigma := \big\{ \gamma \in \overline{\mathbb{R}}: (\Theta, \Phi, T^+, T^-) \text{ does not admit an exponential dichotomy with growth rate } \gamma \big\}\,. $$
The corresponding \emph{resolvent set} is defined by $ \rho := \overline{\mathbb{R}} \setminus \Sigma$.
\end{definition}
In Theorem~\ref{Dichtheorem} below, we characterize the dichotomy spectrum as a disjoint union of at least one and at most $d$ closed intervals. The proof uses a couple of lemmas about the resolvent set that are shown in the following, similarly to \cite{cdlr16} and \cite{mr10}. First, we show that the ranks of invariant projectors are monotonically increasing with respect to the growth rate.
\begin{lemma} \label{gamma1gamma2}
Consider the resolvent set $\rho$ of the linear random dynamical system $(\Theta, \Phi, T^+, T^-)$, and let $\gamma_1, \gamma_2 \in \rho \cap \mathbb{R}$ such that $\gamma_1 \leq \gamma_2$. Choose invariant projectors $P_{\gamma_1}$ and $P_{\gamma_2}$ for the corresponding exponential dichotomies with growth rates $\gamma_1$ and $\gamma_2$. Then we have $\rk P_{\gamma_1} \leq \rk P_{\gamma_2}$. In particular, if $\gamma_1 = \gamma_2$, then $\rk P_{\gamma_1} = \rk P_{\gamma_2}$.
\end{lemma}
\begin{proof}
Let $K_1 \geq 1$ and $K_2 \geq 1, \alpha_2 > 0$ be the corresponding constants for the exponential dichotomies with $\gamma_1$ and $\gamma_2$ respectively. Choose $t^*>0$ large enough such that $K_2 K_1 e^{- \alpha_2 t^*} < 1$, and for fixed $x \in E$, choose $\omega\in \Omega_x = \big\{ \omega \in \Omega \,:\, T^+(\omega,x) > t^*\big\}$. Now let $v(\omega, x)\in \mathcal{N}(P_{\gamma_2}(\omega,x)) \cap \mathcal{R}(P_{\gamma_1}(\omega,x))$ and assume that $v(\omega,x) \neq 0$.  We observe that for all $t \in ( t^* , T^{+}(\omega,x))$, we have
\begin{align*}
\| v(\omega,x) \| &= \| \Phi(-t, \Theta_{t} (\omega,x)) \Phi(t, \omega,x) (\Id - P_{\gamma_2}(\omega,x))v(\omega,x) \| \\
&= \| \Phi(-t, \Theta_{t} (\omega,x)) (\Id - P_{\gamma_2}(\Theta_{t} (\omega,x)))\Phi(t, \omega,x) v(\omega,x) \| \\
&  \leq K_2 e^{-(\gamma_2 + \alpha_2)t} \| \Phi(t, \omega,x) v(\omega,x) \|\leq K_2 K_1 e^{-(\gamma_2 + \alpha_2)t} e^{\gamma_1 t} \|v(\omega,x)\| \\
 &\leq K_2 K_1 e^{- \alpha_2 t} \|v(\omega,x)\| < \|v(\omega,x)\|\,,
\end{align*}
which is a contradiction. Hence,
$$ \mathcal{N}(P_{\gamma_2}(\omega,x)) \cap \mathcal{R}(P_{\gamma_1}(\omega,x)) = \{0\} \fa  x \in E \mbox{ and }\omega \in \Omega_x\,.$$
Since $\Omega_x$ has positive probability and the dimensions of the ranges and null spaces of invariant projectors are constant almost surely, we can deduce that
\begin{equation*}
0 = \dim ( \mathcal{R}(P_{\gamma_1}) \cap \mathcal{N}(P_{\gamma_2}) ) = \rk P_{\gamma_1} + \dim \mathcal{N}(P_{\gamma_2}) - \dim ( \mathcal{R}(P_{\gamma_1}) + \mathcal{N}(P_{\gamma_2}))\,.
\end{equation*}
This is used to observe that
\begin{equation*}
\rk P_{\gamma_2} = d - \dim \mathcal{N}P_{\gamma_2}
= \rk P_{\gamma_1} + d - \dim ( \mathcal{R}P_{\gamma_1} + \mathcal{N}P_{\gamma_2}) \geq \rk P_{\gamma_1}\,.
\end{equation*}
which shows the first statement of the lemma. The second statement is an immediate consequence.
\end{proof}
We proceed with showing that the resolvent set is open in $\mathbb{R}$.
\begin{lemma} \label{DS_closed}
Consider the resolvent set $\rho$ of $(\Theta, \Phi, T^+, T^-)$. Then for all $ \gamma \in \rho \cap \mathbb{R}$, there is an $\epsilon > 0$ such that $B_{\epsilon} (\gamma) \subset \rho \cap \mathbb{R}$, which means that $\rho \cap \mathbb{R}$ is an open set. Furthermore, we have $\rk P_{\zeta} = \rk P_{\gamma}$ for all $\zeta \in B_{\epsilon}(\gamma)$ and every invariant projector $P_{\gamma}$ and $P_{\zeta}$ of the exponential dichotomies of $(\Theta, \Phi, T^+, T^-)$ with growth rates $\gamma$ and $\zeta$, respectively.
\end{lemma}
\begin{proof}
Let  $\gamma \in \rho \cap \mathbb{R}$ and $\alpha, K$ be the constants for the exponential dichotomy with growth rate $\gamma$ and invariant projector $P_{\gamma}$. Set $\epsilon: = \frac{1}{2}\alpha$ and choose any $\zeta \in B_{\epsilon}(\gamma)$. Then for almost all $\omega \in \Omega$ and all $x \in E$,
\begin{align*}
\| \Phi(t, \omega,x)P_{\gamma}(\omega,x) \| &\leq K e^{(\zeta - \frac{1}{2}\alpha)t}\fa 0 \leq t < T^+(\omega,x)\,, \\
\| \Phi(t, \omega,x)(\Id - P_{\gamma}(\omega,x)) \| &\leq K e^{(\zeta + \frac{1}{2}\alpha)t}\fa 0 \geq t > T^-(\omega,x)\,.
\end{align*}
Hence, $P_{\gamma}$ is an invariant projector for the exponential dichotomy with growth rate $\zeta$ and we have $\rk P_{\zeta} = \rk P_{\gamma}$ for any other such invariant projector by Lemma~\ref{gamma1gamma2}.
\end{proof}
The last ingredient for proving Theorem~\ref{Dichtheorem} is the following lemma.
\begin{lemma} \label{rankequal}
Consider the resolvent set $\rho$ of $(\Theta, \Phi, T^+, T^-)$ and let $\gamma_1, \gamma_2 \in \rho \cap \mathbb{R}$ such that $\gamma_1 < \gamma_2$. Choose invariant projectors $P_{\gamma_1}$ and $P_{\gamma_2}$ for the exponential dichotomies with growth rates $\gamma_1$ or $\gamma_2$, respectively. Then $[\gamma_1, \gamma_2] \subset \rho$ if and only if $\rk P_{\gamma_1} = \rk P_{\gamma_2}$.
\end{lemma}
\begin{proof}
First assume that $[\gamma_1, \gamma_2] \subset \rho$, and suppose for contradiction that $\rk P_{\gamma_1} \neq \rk P_{\gamma_2}$. Choosing invariant projectors $P_{\gamma}$ for exponential dichotomies with growth rates $\gamma \in (\gamma_1, \gamma_2)$, we define
$$ \zeta_0 := \sup \big\{ \zeta \in [\gamma_1, \gamma_2]\,:\, \rk P_{\zeta} \neq \rk P_{\gamma_2} \big\}\,.$$
However, according to Lemma~\ref{DS_closed} there is an $\epsilon > 0$ such that $ \rk P_{\zeta} = \rk P_{\zeta_0}$ for all $\zeta \in B_{\epsilon}(\zeta_0)$, contradicting the definition of $\zeta_0$. Hence, we have shown the first implication.

Assume now that $\rk P_{\gamma_1} = \rk P_{\gamma_2}$. We have already seen in the proof of Lemma~\ref{gamma1gamma2} that
\begin{equation}\label{rel4}
  \mathcal{N}(P_{\gamma_2}(\omega,x)) \cap \mathcal{R}(P_{\gamma_1}(\omega,x)) = \{0\} \fa  x \in E \mbox{ and } \omega \in \Omega_x\,,
\end{equation}
where $\Omega_x = \{ \omega \in \Omega \,:\, T^+(\omega,x) > t^*\}$ has positive probability. Define $\widetilde W:= \{(\omega,x): \omega\in\Omega_x\}$, and note that \eqref{rel4} holds for all $(\omega,x)\in \widetilde W$.

Define $\widehat W:=\{(\omega,x): T^+(\omega, x)-T^-(\omega,x) \le t^*\}$. It is clear that $\widehat W$ and $(E\times \Omega) \setminus \widehat W$ are invariant with respect to $\Theta$, and it is easy to see that \eqref{rel4} holds for all $(\omega,x)\in (E\times \Omega) \setminus \widehat W$ (note that $\widetilde W\subset (E\times \Omega) \setminus \widehat W$).
Due to $\rk P_{\gamma_1} = \rk P_{\gamma_2}$, we can define an projector $P$ on $(E\times \Omega) \setminus \widehat W$ such that
$\mathcal{N}(P(\omega,x)) = \mathcal{N}(P_{\gamma_2}(\omega,x))$ and $\mathcal{R}(P(\omega,x)) = \mathcal{R}(P_{\gamma_1}(\omega,x))$ for all $(\omega,x)\in (E\times \Omega) \setminus \widehat W$.
This means that for constants $K_1,\alpha_1, K_2, \alpha_2$ we have for all $(\omega,x)\in (E\times \Omega) \setminus \widehat W$ that
\begin{align*}
\| \Phi(t, \omega,x)P(\omega,x) \| &\leq K_1 e^{(\gamma_1 - \alpha_1)t} \fa 0 \leq t < T^+(\omega,x)\,, \\
\| \Phi(t, \omega,x)(\Id - P(\omega,x)) \| &\leq K_2 e^{(\gamma_2 + \alpha_2)t} \fa  0 \geq t > T^-(\omega,x)\,.
\end{align*}
Setting $K:= \max \{K_1,K_2\}$ and $\alpha:= \min\{\alpha_1, \alpha_2\}$, we obtain for all $\gamma \in [\gamma_1, \gamma_2]$
and $(\omega,x)\in (E\times \Omega) \setminus \widehat W$ that
\begin{align} \label{expdichex}
\begin{array}{r@{\;\,\leq\;\,}l}
\| \Phi(t, \omega,x)P(\omega,x)\| & K e^{(\gamma - \alpha)t} \fa 0 \leq t < T^+(\omega,x)\,, \\
\| \Phi(t, \omega,x)(\Id - P(\omega,x)) \| & K e^{(\gamma + \alpha)t}\fa 0 \geq t > T^-(\omega,x)\,.\end{array}
\end{align}
We define $P \equiv P_{\gamma_1}$ on the remaining part $\widehat W$. Note that due to construction, the projector $P$ is invariant. Note that due to Assumption~\ref{assds}, and by possibly enlarging the constant $K$, the above estimate \eqref{expdichex} holds also on $\widehat{W}$.
This implies that $[\gamma_1, \gamma_2] \subset \rho$.
\end{proof}
For $a \in \R$, we define $[-\infty, a] := (- \infty, a] \cup \{-\infty\}$, $[a,\infty] :=  [a, \infty) \cup \{\infty\}$, $[-\infty, -\infty] := \{-\infty\}$, $[\infty,\infty] :=  \{\infty\}$ and $[-\infty, \infty] := \overline{\R}$.
Analogously to the case without killing \cite[Theorem 3.4]{cdlr16}, we can now prove the following Spectral Theorem.
\begin{theorem}[Spectral Theorem] \label{Dichtheorem}
Consider the linear system $(\Theta, \Phi, T^+, T^-)$ with dichotomy spectrum $\Sigma$. Then there exists an $n \in \{1, \dots, d \}$ such that
$$ \Sigma = [a_1,b_1] \cup \dots \cup [a_n, b_n]\,,$$
where $ - \infty \leq a_1 \leq b_1 < a_2 \leq b_2 < \dots < a_n \leq b_n \leq \infty$.
\end{theorem}
\begin{proof}
Since $\rho \cap \mathbb{R}$ is open according to Lemma~\ref{DS_closed}, the set $\Sigma \cap \mathbb{R}$ is the disjoint union of closed intervals. Furthermore, by Lemma~\ref{P_infty}, $(-\infty, b_1] \subset \Sigma$ implies $[-\infty, b_1] \subset \Sigma$, and $[a_n, \infty) \subset \Sigma$ implies  $[a_n, \infty] \subset \Sigma$.

The fact that $1\leq n \leq d$ is a consequence of Lemma~\ref{rankequal} which can now be derived exactly as in the proof of \cite[Theorem 3.4]{cdlr16}.
\end{proof}

In addition, we can prove an analogue to \cite[Theorem 4.5]{cdlr16} for the case of absorption at the boundary, relating the upper and lower limits of finite-time Lyapunov exponents to the extremal values of the dichotomy spectrum.
%
\begin{theorem}[Supremum and infimum of the dichotomy spectrum] \label{boundaryDichotomy}
Let $\Sigma$ denote the dichotomy spectrum of $(\Theta, \Phi, T^+, T^-)$.  Recall that the finite-time Lyapunov exponent for each $(\omega,x) \in \Omega \times E$ and $v \in \mathbb{R}^d \setminus \{0\}$ is given by
\begin{equation*}
\lambda_v(t,\omega,x) = \frac{1}{t} \ln \frac{ \|\Phi(t,\omega,x) v \|}{\|v\|} \fa t\in\big(0, T^+(\omega,x)\big)\,.
\end{equation*}
Then
\begin{equation}\label{supds}
 \lim_{t \to \infty} \sup_{x\in E}\esssup_{ \{\omega\in\Omega:T^+(\omega,x) > t\}} \sup_{ v \not=0} \lambda_v(t, \omega, x) = \sup \Sigma\,,
\end{equation}
provided that $\sup \Sigma < \infty$ and
\begin{equation}\label{supds2} \lim_{t \to \infty}
\inf_{x\in E}\essinf_{ \{\omega\in\Omega:T^+(\omega,x) > t\}} \inf_{ v \not=0}
\lambda_v(t, \omega, x) = \inf \Sigma
\end{equation}
provided that $\inf \Sigma > -\infty$.
\end{theorem}
\begin{proof}
By definition of $\lambda_v(t, \omega, x)$ we get for all $t,s \geq 0$
\begin{align*}
 (t+s)  &  \sup_{x\in E}\esssup_{ \{\omega\in\Omega:T^+(\omega,x) > t+s\}} \sup_{ v \not=0} \lambda_v(t+s, \omega, x) \\
&\leq t  \sup_{x\in E}\esssup_{ \{\omega\in\Omega:T^+(\omega,x) > t\}} \sup_{ v \not=0}\lambda_v(t, \omega, x) + s  \sup_{x\in E}\esssup_{ \{\omega\in\Omega:T^+(\omega,x) > s\}} \sup_{ v \not=0} \lambda_v(s, \omega, x)\,.
\end{align*}
This implies that the function  $$t\mapsto  t  \sup_{x\in E}\esssup_{ \{\omega\in\Omega:T^+(\omega,x) > t\}} \sup_{ v \not=0}\lambda_v(t, \omega, x)$$ is subadditive.
Hence, we obtain
$$ \lim_{t \to \infty}  \sup_{x\in E}\esssup_{ \{\omega\in\Omega:T^+(\omega,x) > t\}} \sup_{ v \not=0} \lambda_v(t, \omega, x) = \limsup_{t \to \infty}  \sup_{x\in E}\esssup_{ \{\omega\in\Omega:T^+(\omega,x) > t\}} \sup_{ v \not=0} \lambda_v(t, \omega, x).$$
Provided $\sup\Sigma < \infty$, we show that
$$ \gamma := \limsup_{t \to \infty}  \sup_{x\in E}\esssup_{ \{\omega\in\Omega:T^+(\omega,x) > t\}} \sup_{ v \not=0}\lambda_v(t, \omega, x) = \sup \Sigma.$$
Since $\sup \Sigma < \infty$, there exists a $K \geq 1$ such that for almost all $\omega \in \Omega$ and all $x \in E$, $v  \in \mathbb{R}^d \setminus \{0\}$
\begin{equation} \label{sigmabound}
 \| \Phi(t,\omega,x)v\| \leq K e^{t \sup \Sigma} \|v\| \fa 0 \leq t < T^+(\omega,x).
\end{equation}
Assume for contradiction that $\gamma < \sup \Sigma$.
From the definition of $\gamma$, this means that there exists a $t_0 > 0$ such that for almost all $\omega \in \Omega$ and all $x \in E$ with $T^+(\omega,x) > t_0$ and all $v  \in \mathbb{R}^d \setminus \{0\}$,
\begin{equation*}
 \| \Phi(t,\omega,x) v\| \leq K \exp\big(\tfrac{1}{2}t (\gamma + \sup \Sigma)\big) \|v\| \fa t_0 \leq t <T^+(\omega,x)\,.
\end{equation*}
Hence, together with \eqref{sigmabound}, we obtain for almost all $\omega \in \Omega$ and all $x \in E$, $v  \in \mathbb{R}^d \setminus \{0\}$ that
\begin{equation*}
 \| \Phi(t,\omega,x) v\| \leq \hat{K} \exp\big(\tfrac{1}{2}t (\gamma + \sup \Sigma)\big) \|v\| \fa 0 \leq t <  T^+(\omega,x)\,,
\end{equation*}
where
$$ \hat{K} := \max \left\{1, K \exp\big(\tfrac{t_0}{2}(\sup \Sigma - \gamma)\big)\right\}. $$
By the definition of $\Sigma$, this implies that $\sup \Sigma \leq \frac{1}{2}(\gamma + \sup \Sigma)$, which is a contradiction. Hence, we have established that $\gamma \geq \sup \Sigma$.

Assume now that $\gamma > \sup \Sigma$, which implies $\sup \Sigma < \infty$. Hence, by the definition of the dichotomy spectrum, there exists a $K \geq 1$ such that for almost all $\omega \in \Omega$ and all $x \in E$, $v  \in \mathbb{R}^d \setminus \{0\}$,
\begin{equation*}
 \| \Phi(t,\omega,x) v\| \leq K \exp\big(\tfrac{1}{2}t (\gamma + \sup \Sigma)\big)\|v\| \fa 0 \leq t < T^+(\omega,x)\,.
\end{equation*}
On the other hand, this yields
$$\lambda_v(t, \omega,x) \leq \frac{\ln K}{t} + \frac{1}{2}(\gamma + \sup \Sigma)$$
for all $v \in \mathbb{R}^d \setminus \{0\}$ whenever $ t < T^+(\omega,x)$. Since $\frac{\ln K}{t} \to 0$ as $ t \to \infty$, we conclude that
$$ \gamma = \limsup_{t \to \infty}  \sup_{x\in E}\esssup_{ \{\omega\in\Omega:T^+(\omega,x) > t\}} \sup_{ v \not=0} \lambda_v(t, \omega, x) \leq \frac{1}{2} (\gamma + \sup \Sigma)\,, $$
which is again a contradiction. This proves the equality \eqref{supds}, and the second equality \eqref{supds2} follows analogously.
\end{proof}

\section*{Acknowledgments} The authors would like to thank Nils Berglund, Martin Hairer, Christian K{\"u}hn, Nikolas N{\"u}sken and Lai-Sang Young for very useful discussions. Maximilian Engel was supported by a Roth Scholarship from the Department of Mathematics at Imperial College London and the SFB Transregio 109 "Discretization in Geometry and Dynamcis" sponsored by the German Research Foundation (DFG). Jeroen S.W.~Lamb acknowledges the support by Nizhny Novgorod University through the grant RNF 14-41-00044, and Martin Rasmussen was supported by an EPSRC Career Acceleration Fellowship EP/I004165/1. This research has also been supported by EU Marie-Curie IRSES Brazilian-European Partnership in Dynamical Systems (FP7-PEOPLE-2012-IRSES 318999 BREUDS) and EU Marie-Sk\l odowska-Curie ITN Critical Transitions in Complex Systems (H2020-MSCA-2014-ITN 643073 CRITICS).




\end{document}